\documentclass[12pt]{amsart}
\usepackage{amssymb,amsthm,amsmath,amsfonts}

\newlength{\tmpl}
\newcommand{\boxstar}{
\setlength{\tmpl}{\fboxsep}
\setlength{\fboxsep}{0pt}
\,\fbox{$\star$}\,
\setlength{\fboxsep}{\tmpl}
}

\usepackage{hyperref}
\usepackage{pstricks}
\usepackage[margin=2cm,a4paper]{geometry}
\usepackage {tikz}

\usepackage[extra]{tipa}
\usepackage{graphicx}
\usepackage{pst-node}
\usepackage{enumerate}

\numberwithin{equation}{section}

\newcommand{\onetwo}[1]{\hat{1}_2^{#1}} 
\newcommand{\pispecial}[1]{\nu_{0#1}} 

\DeclareMathOperator{\diag}{\mathrm{diag}}
\DeclareMathOperator{\Tr}{\mathrm{Tr}}

\usepackage{mathtools}

\newcommand{\Krewl}[1]{\mathpalette\harpleft{#1}}
\newcommand{\harpleftsign}{\scriptstyle\leftharpoonup}
\newcommand{\harpleft}[2]{%
  \ifx\displaystyle#1\doalign{$\harpleftsign$}{#1#2}\fi
  \ifx\textstyle#1\doalign{$\harpleftsign$}{#1#2}\fi
  \ifx\scriptstyle#1\doalign{\scalebox{.6}[.9]{$\harpleftsign$}}{#1#2}\fi
  \ifx\scriptscriptstyle#1\doalign{\scalebox{.5}[.8]{$\harpleftsign$}}{#1#2}\fi
}

\newcommand{\Krewr}[1]{\mathpalette\harpright{#1}}
\newcommand{\harprightsign}{\scriptstyle\rightharpoonup}
\newcommand{\harpright}[2]{%
  \ifx\displaystyle#1\doalign{$\harprightsign$}{#1#2}\fi
  \ifx\textstyle#1\doalign{$\harprightsign$}{#1#2}\fi
  \ifx\scriptstyle#1\doalign{\scalebox{.6}[.9]{$\harprightsign$}}{#1#2}\fi
  \ifx\scriptscriptstyle#1\doalign{\scalebox{.5}[.8]{$\harprightsign$}}{#1#2}\fi
}

\newcommand{\doalign}[2]{%
 {\vbox{\offinterlineskip\ialign{\hfil##\hfil\cr#1\cr$#2$\cr}}}%
}

\usepackage{yhmath}
\newcommand{\Krewlr}[1]{ \widetriangle{#1}} 

\newcommand{\ED}[1][{}]{E^{\mathcal D}_{#1}}

\def\R{{\mathbb R}}

\def\C{{\mathbb C}}
\def\IC{{\mathbb C}}
\def\IR{{\mathbb R}}
\def\N{{\mathbb N}}
\def\bX{{\boldsymbol{X}}}
\def\bZ{{\boldsymbol{Z}}}
\def\<{\langle}
\def\>{\rangle}

\def\Q{{\mathnormal Q}}

\def\X{{\mathnormal X}}

\def\Y{{\mathnormal Y}}

\def\A{\mathcal{A}}

\newcommand{\norm}[1]{\lVert#1\rVert}
\newcommand{\abs}[1]{\lvert#1\rvert}
\newcommand{\One}{\mathbf{1}}
\DeclareMathOperator{\Var}{\mathrm{Var}} 
\DeclareMathOperator{\SG}{\mathfrak{S}} 
\DeclareMathOperator{\NC}{\mathit{NC}} %
\DeclareMathOperator{\SP}{\mathcal{P}} %
\DeclareMathOperator{\NCeven}{\mathit{NCE}} %

\newtheorem{th-def}{Theorem-Definition}[section]
\newtheorem{theo}{Theorem}[section]
\newtheorem{lemm}[theo]{Lemma}
\newtheorem{Con}[theo]{Conjecture}
\newtheorem{prop}[theo]{Proposition}
\newtheorem{cor}[theo]{Corollary}
\theoremstyle{definition}
\newtheorem{problem}[theo]{Problem}
\newtheorem{defi}[theo]{Definition}

\newtheorem{Rem}[theo]{Remark}

\def\cvput#1[#2]{\pnode(#1,1){#1} \pscircle*(#1,1){.1} \rput(#1,.5){$#2$}}

\title{Sample Variance in Free Probability }
\author[Wiktor Ejsmont]{Wiktor Ejsmont}

\address[Wiktor Ejsmont]{ Wydzia\l  Matematyki i Informatyki, Uniwersytet Im. Adama Mickiewicza, Collegium
Mathematicum, Umultowska 87, 61-614 Pozna\'n , Poland}
\email{wiktor.ejsmont@gmail.com}
\author[Franz Lehner]{Franz Lehner}
\address[Franz Lehner]{Department of Discrete Mathematics
TU Graz
Steyrergasse 30, 8010 Graz, Austria }
\email{lehner@math.tugraz.at}
\subjclass[2010]{Primary: 46L54. Secondary: 62E10.}

\keywords{Sample variance, Wigner semicircle law, free Poisson distribution, free infinite divisibility, free cumulants, noncrossing partitions}

\begin{document}

\begin{abstract} 
Let $\X_1, \X_2,\dots, \X_n$ denote i.i.d.~centered standard normal random variables,
then the law of the sample variance 
$Q_n=\sum_{i=1}^n(\X_i-\overline{\X})^2$ is the $\chi^2$-distribution with
$n-1$ degrees of freedom.
It is an open problem in classical probability
to characterize all distributions with this property and in particular, 
whether it characterizes the normal law.
In this paper we present a solution of the free analogue of this question
and show that the only distributions, 
whose free sample variance is distributed according to a free
$\chi^2$-distribution, are the semicircle law and more generally
so-called \emph{odd} laws, by which we mean laws with vanishing
higher order even cumulants. In the way of proof we derive an
explicit formula for the free cumulants of $Q_n$ which shows that
indeed the odd cumulants do not contribute and which exhibits
an interesting connection to the concept of $R$-cyclicity.
\end{abstract}
\date{\today}
\setlength{\parindent}{0pt}
\maketitle
\begin{verse} \textbf{
Dedicated to our friend and mentor 
Marek Bo\.zejko
on the occasion of his 70-th birthday}
\end{verse}

\section{Introduction}
\noindent 
Many questions in classical statistics involve characterization problems,
which usually are instances of the following very general question:

\begin{problem}
Let $\X_1, \X_2,\dots, \X_n$ be independent random variables with common
unknown distribution function $F$, and $T :=
T(\X_1, \X_2,\dots, \X_n)$ a statistic, based on
$\X_1, \X_2,\dots, \X_n$, with distribution function $G$. 
Can $F$ be recovered from $G$? 
\end{problem}

Problems of this kind are the central leitmotiv of the fundamental work of
Kagan,  Linnik and Rao \cite{KaganLinnikRao:1973}.
In the present paper we solve the free version of the following problem,
which is still open in classical probability and might be called
$\chi^2$-conjecture,
see \cite[p.~466]{KaganLinnikRao:1973}:

\begin{Con}
\label{conj:KaganLinnikRao}
 If $\X_1, \X_2,\dots, \X_n$ are non-degenerate, independently and identically distributed  classical random
variables with finite non-zero variance $\sigma^2$, then a necessary and
sufficient condition for $\X_1$ to be normal is that
$\sum_{i=1}^n(\X_i-\overline{\X})^2/\sigma^2$ be distributed as classical
chi-square distribution with $n -1$ degrees of freedom. 
\end{Con}

The classical $\chi^2$-conjecture was studied previously by several authors. 
The first result is due to Ruben \cite{Ruben:1974},
who proved the conjecture under the assumption that either $n=2$ or
$\X_1$ is symmetric. 
It is not known whether the symmetry  hypothesis
can be dropped for $n\geq 3$. 
In a later paper \cite{Ruben:1975} Ruben  used combinatorial tools
to show  that the symmetry condition can be dropped
provided the sum of squares of the sample observations about the sample
mean, divided by $\sigma^2$, is distributed as chi-square for two distinct
sample sizes $m\neq n$ and $m,n\geq 2 $. 
The proof given by Ruben is based on the cumulants of the sample variance
and is somewhat complicated. 
Shortly later a simpler and more direct proof 
based on the moments of the sample variance 
was presented by Bondesson \cite{Bondesson:1977}.

The original problem was  solved recently by  Golikova and Kruglov
\cite{GolikovaKruglov:2015} 
under the additional assumption that $\X_1, \X_2,\dots, \X_n$ 
are independent infinitely divisible random variables. 

The following related characterization problem
was solved by Kagan and Letac~\cite{KaganLetac:1999}:
Let
$\X_1 ,\X_2 ,\dots, \X_n$ 
be independent and identically distributed random variables
and assume that the distribution of the quadratic statistic
 $\sum_{i=1}^n(\X_i-\overline{\X}+a_i)^2$ depends only on $\sum_{i=1}^na_i^2$.
Then each $\X_i$ have distribution $N(0, \sigma)$.

In the present paper we answer analogous questions in free probability. 
Free probability and free convolution
was introduced by Voiculescu in \cite{Voiculescu:1985}
as a tool to study the von Neumann algebras of free groups.
Free probability is now an established field of research with deep connections
to combinatorics, random matrix theory, representation theory
and many analogies to classical probability.
Let us restrict our discussion to two specific ones, which are relevant
to the problems discussed in the present paper.
On the analytic side the Bercovici-Pata bijection
\cite{BercoviciPata:1999} provides a one-to-one correspondence
between infinitely divisible measures with respect to classical and
free convolution. For example, the  analogue of the normal law is played 
by Wigner's semicircle distribution which features as the limit law
in the free central limit theorem.

On the combinatorial side we will make heavy use of
free cumulants introduced by Speicher~\cite{Speicher:1994}.
Roughly speaking, any result about classical cumulants can be
translated to free probability by replacing the lattice of set partitions
by the lattice of noncrossing partitions.
Our standard reference for free cumulants is the book~\cite{NicaSpeicher:2006}.

We are concerned here with free analogues of characterization theorems 
in the spirit of \cite {KaganLinnikRao:1973}.
The study of free analogues of classical theorems has witnessed
increasing interest during the last decade, see, e.g.,
\cite{BozejkoBryc:2006,Ejsmont:2013,HiwatashiKurodaNagisaYoshida:1999,Lehner:2004,SzpojankowskiWesolowski:2014,Szpojankowski:2014,Szpojankowski:2015}. 
Many properties of free random variables are analogous to those of their
classical counterparts, in particular when they are picked according
to the Bercovici-Pata bijection. There are, however, exceptions, mostly
due to the failure of Marcinkiewicz' and Cram\'er's theorems in free probability.
In particular, Bercovici and Voiculescu \cite{BercoviciVoiculescu:1995}  
showed that there exist free random variables
with a finite number of nonvanishing free cumulants which are not semicircular,
see \cite{ChistyakovGoetze:2011} for a characterization of such distributions.
This class of distributions appears in some (but not all)
free characterization problems which are analogues of classical
characterizations of the normal law,
cf.~\cite{Lehner:2003,ChistyakovGoetze:2011}.

In the present paper we show that 
Conjecture~\ref{conj:KaganLinnikRao} also falls in this class of problems
and instead of Wigner laws we obtain the class of \emph{odd} laws,
i.e., laws with vanishing even cumulants. 
Such laws do not exist in classical probability, but
can be constructed in free probability using the results of
\cite{ChistyakovGoetze:2011}. On the way we encounter a remarkable
cancellation phenomenon: odd cumulants do not contribute to the distributions
of certain quadratic statistics.

The paper is organized as follows. In section 2 we review basic free probability  and the statement of the main result. Next in the subsection 2.2 we quote complementary facts, lemmas and indications. In the  third section we  prove our main results. 
Finally, in section 4  
we look more closely at the relation between the sample variance, the free commutator, $R-$cyclic matrices and free infinite divisibility. 


\section{Free   probability  and statement of the main result}

\subsection{Basic Notation and Terminology}
A tracial noncommutative probability space is a pair $(\mathcal{A},\tau)$
where $\mathcal{A}$ is a von Neumann algebra, and  $\tau:\mathcal{A} \to
\IC$ is a normal, faithful, tracial state, i.e., $\tau$ is linear and
continuous in the weak* topology, $\tau(\X \Y)=\tau(\Y \X)$,  $\tau(\mathrm{I})=1$, $\tau(\X\X^*)\geq 0$ and $\tau(\X \X^{*}) = 0$ implies $\X = 0$ for all $\X,\Y \in \mathcal{A}$.

The (usually taken to be self-adjoint) elements $\X\in{\mathcal{A}}_{sa}$
are called (noncommutative) random variables.
Given a noncommutative random variable $\X\in{\mathcal{A}}_{sa}$, 
the distribution of $\X$ in the state $\tau$ is the unique probability measure
$\mu_\X$ on $\IR$ (given by the spectral theorem) which reproduces its moments, i.e.,
 $\tau( X^n)=\int_{\IR} \lambda^n\,d\mu_\X(\lambda)$
for $n\in\N$. This definition can be extended to self-adjoint possibly
unbounded operators $\X$ affiliated to $\mathcal{A}$ by requiring that
 $\tau( f(\X))=\int_{\IR} f(\lambda)\,d\mu_\X(\lambda)$
for any bounded Borel function $f$ on $\IR$.
The set of affiliated operators is denoted by $\widetilde{\mathcal{A}}_{sa}$.

\subsection{Free Independence, Free Convolution and Free infinite Divisibility}
A family of von Neumann subalgebras $\left(\mathcal{A}_i\right)_{i\in I}$ 
of $\mathcal{A}$ 
are called \emph{free}
 if $\tau(\X_{1} \dots \X_{n} ) = 0$ whenever $\tau(\X_{j} ) = 0$ for all
$j = 1,\dots, n$ and $\X_{j} \in \mathcal{A}_{i(j)}$ for some indices $i(1)\neq i(2)\neq \dots \neq i(n)$.
Random variables $\X_{1},\dots ,\X_{n} $  are freely independent (free) if the
subalgebras they generate are free. 
Free random variables can be constructed using the reduced free product
of von Neumann algebras \cite{Voiculescu:1985}.
For more details about free convolutions and free probability theory, the reader can consult \cite{NicaSpeicher:2006,VoiculescuDykemaNica:1992}.

It can be shown that the joint distribution of free random variables $X_i$
is uniquely determined by the distributions of the individual random variables
$X_i$ and therefore the operation of \emph{free convolution} is well defined:
Let $\mu$ and $\nu$ be probability measures on $\IR$, and
$\X,\Y$ self-adjoint free random variables with respective distributions
$\mu$ and $\nu$,
The distribution of $\X+\Y$ is called the free additive convolution of $\mu$
and $\nu$ and is denoted by $\mu \boxplus \nu$.

In analogy with classical probability,
a probability measure $\mu$ on $\R$ is said to be
\emph{freely infinitely divisible} (or FID for short)
if for each $n \in \{1, 2, 3, \dots \}$ there exists a probability measure
$\mu_n$ such that
$\mu=\underbrace{ \mu_n\boxplus\dots\boxplus\mu_n}_{n-times}$.

\subsection{The Cauchy-Stieltjes Transform and Free Convolution}
The analytic approach to free convolution uses the Cauchy transform
\begin{align}
\label{eq:hm:ball}
G_\mu(z)=\int_{\IR}\frac{1}{z-y}\mu(dy). 
\end{align}
of a probability measure $\mu$. It is analytic on the upper half
plane $\IC^+=\{x+iy|x,y\in \IR, y>0\}$ 
and takes values in the closed lower half plane
 $\IC^-\cup\IR$. 
The Cauchy transform has an inverse at a neighbourhood of infinity
which has the form
$$
G_\mu^{-1}(z) = \frac{1}{z} + R_\mu(z)
$$
where $R_\mu(z)$ is analytic in a neighbourhood of zero and is called 
\emph{$R$-transform}.
Then free convolution is defined (see~\cite{Voiculescu:1986})
via the identity
\begin{align} \label{freeconv}
R_{\mu\boxplus\nu}=R_\mu+R_\nu.
\end{align} 

The coefficients of the $R$-transform
\begin{align} \label{rtr}
R_{\X}(z)=\sum_{n=0}^{\infty}\,K_{n+1}(\X)\,z^n.
\end{align} 
are called \emph{free cumulants} of the random variable $\X$.



The Cauchy transform is related to the
moment generating function $M_{\X}$ as follows:
\begin{align}\label{mgf}
M_{\X}(z)=\sum_{n=0}^{\infty}\,\tau(\X^n)\,z^n = \frac{1}{z}G_\X\left(\frac{1}{z}\right).
\end{align}
\subsection{Some probability distributions}
Let us now recall basic properties of some specific probability distributions
which play prominent roles in the present paper.

\subsubsection{Wigner semicircular distribution}
A non-commutative random variable $\X$ is said to be free normal variable
(i.e. have Wigner semicircular distribution) if the Cauchy-Stieltjes transform
is given by the formula
\begin{eqnarray}
G_{\mu}(z)=\frac{z  -\sqrt{z ^2 - 4}}{2} , \label{eq:GtransformataMixner}
\end{eqnarray}
where  $|z|$ is big enough, where the branch of the analytic square root should be determined by the condition
that $\Im(z)>0\Rightarrow \Im(G_\mu(z))\leqslant 0$ (see \cite{SaitohYoshida:2001}). 
Equation (\ref{eq:GtransformataMixner}) describes the family of distributions  with mean zero and variance one (see \cite{Ejsmont:2012,SaitohYoshida:2001}). 
This measure has density 
$$\frac{\sqrt{4-x^2}}{2\pi},$$
on $-2 \leq x \leq 2 $. 
The Wigner semicircular distribution have cumulants $
K_i=0 
$
 for $i>2$.  

\subsubsection{Free Poisson distribution}
A non-commutative random variable $\X$ is said to be free-Poisson variable if it has Marchenko-Pastur (or free-Poisson) distribution $\nu=\nu(\lambda,\alpha)$ defined by the formula
\begin{align} \label{MPdist}
\nu=\max\{0,\,1-\lambda\}\,\delta_0+ \tilde{\nu},
\end{align}
where $\lambda\ge 0$ and the measure $\tilde{\nu}$, supported on the interval $(\alpha(1-\sqrt{\lambda})^2,\,\alpha(1+\sqrt{\lambda})^2)$, $\alpha>0$ has the density (with respect to the Lebesgue measure)
$$
\tilde{\nu}(dx)=\frac{1}{2\pi\alpha x}\,\sqrt{4\lambda\alpha^2-(x-\alpha(1+\lambda))^2}\,dx.
$$
The parameters $\lambda$ and $\alpha$ are called the rate and the jump size, respectively.
It is worth to note that a non-commutative variable with Marchenko-Pastur distribution arises also as a limit in law (in non-commutative sense) of variables with distributions $((1-\frac{\lambda}{N})\delta_0+\frac{\lambda}{N}\delta_{\alpha})^{\boxplus N}$ as $N\to\infty$, see \cite{NicaSpeicher:2006}. Therefore, such variables are often called free-Poisson. It is easy to see that if $X$ is free-Poisson, $\nu(\lambda,\alpha)$, then $K_n(\X)=\alpha^n\lambda$. Therefore its
$R$-transform has the form
$$R(z)=\frac{\lambda\alpha}{1-\alpha z}.$$

\subsubsection{Free  chi-square distribution}
Let $\X_1,\dots,\X_n$ be free identically distributed random variables from the Wigner semicircular distribution with non-zero variance $\sigma^2$ and mean zero, and  $\delta=\sum_{i=1}^n m_i^2$  $(m_i \in \R)$. We call the distribution of the random variable $\sum_{i=1}^n (\X_i+m_i)^2$ the free chi-square distribution with $n$ degrees of freedom
and noncentrality parameter $\delta$, and we denote this distribution $\chi^2(n,\sigma,\delta)$ (a first version of this definition was introduced in \cite{HiwatashiKurodaNagisaYoshida:1999}).
In terms of $R$-transforms, a random variable $Y$ has distribution
$\chi^2(n,\sigma,\delta)$ if and only if
\begin{align} R_Y(z)=\frac{n\sigma^2}{1-\sigma^2z}+\frac{\delta}{(1-2z)^2}. \label{eq:rtransormatachisq}\end{align}
If $\delta=0$, the free chi-square distribution is called central, otherwise non-central and then we will write $\chi^2(n,\sigma)$ and from \eqref{eq:rtransormatachisq} we see that 
 $\chi^2(n,\sigma)$ has the Marchenko-Pastur  distribution $\nu(n,\sigma^2)$.  Moreover,  we will use the notation $\chi^2(n):=\chi^2(n,1)$.
It was shown in \cite{HiwatashiKurodaNagisaYoshida:1999}
that these distributions form a semigroup, namely 
$\chi^2(n_1,\sigma,\delta_1)\boxplus\chi^2(n_2,\sigma,\delta_2)=
\chi^2(n_1+n_2,\sigma,\delta_1+\delta_2)$.

\subsubsection{Even elements}
 We call an element $\X\in \mathcal{A}$ even if  all its odd moments vanish, i.e. $\tau(\X^{2i+1})=0$ for all $i\geq 0.$ 
It is immediately seen that the vanishing of all odd moments is
equivalent to the vanishing of all odd cumulants, i.e., $K_{2i+1}(\X)=0$  
and thus the even cumulants contain the complete information about the distribution of an even element. The sequence
$\alpha_n=K_{2n}(\X)$ is called the \emph{determining sequence} of $X$.

\subsubsection{Odd elements}
 We call an element $\X\in \mathcal{A}$ \emph{odd} if $K_2(\X)>0$ and all its
 even free cumulants of order higher than two vanish, 
i.e. if $K_{2i}(\X) = 0$ for all $i \geq 2$.

The basic example of such a law is Wigner's semicircular distribution.
The classical analogue of odd elements only include the normal distribution
because otherwise we could construct a normal random variable
which is the sum of independent non-normal random variables
(see below for the free case).
This contradicts Cram\'er's decomposition theorem.
However the free analogues of Marcinkiewicz' and Cram\'er's theorems fail.
Bercovici and Voiculescu~\cite{BercoviciVoiculescu:1995} showed that
there exist probability distributions $\mu_\epsilon$ with free cumulants
 $K_1(\X)=0$, $K_2(\X)=1$, $K_3(\X)=\epsilon$ and $K_i(\X)=0$ 
for $i\geq 4 $ if $\epsilon$ is small enough.
This is an odd element and thus an explicit counterexample to the free analogue
of Marcinkiewicz' theorem. To invalidate Cram\'er's theorem,
take free copies $\X_1$ and $\X_2$ of random variables with distribution $\mu_\epsilon$, then the difference $\X_1-\X_2$ is semicircular.
Chistyakov and G\"otze \cite{ChistyakovGoetze:2011} gave a detailed
description of laws with finitely many free cumulants of arbitrary order.
Thus an abundance of odd laws exists.

\subsection{The main result}
The main result of this paper is the following characterization of odd elements
in terms of the sample variance. The proof of this theorem is given in
Section 3.

 \textbf{The sample variance} of a finite sequence of random variables $X_i$
is the quadratic form
\begin{equation}
S^2_n
=\frac{1}{n}\sum_{i=1}^n(\X_i-\overline{\X})^2
=\frac{1}{n}\Big(1-\frac{1}{n}\Big)\sum_{i=1}^n\X_i^2-\frac{1}{n}\sum_{i,j=1, \textrm{ }i\neq j}^n\X_i\X_j
=\frac{1}{n^2}\sum_{1\leq i < j \leq n} (\X_i-\X_j)^2.
\label{eq:SampleVariance}
\end{equation}
However in order to simplify notation in the present paper we chose to
consider and call ``sample variance'' the rescaled quadratic form $Q_n=n S^2_n
=\sum_{i=1}^n(\X_i-\overline{\X})^2
$.

Our main result resolves the free analogue of $\chi^2$-conjecture. 
\begin{theo} 
 Let $\X_1, \X_2,\dots, \X_n \in \A_{sa}$ be free copies of
a random variable $X$ with finite non-zero variance $\sigma^2$.
 Then $\Q_n$
is distributed according to $\chi^2(n-1,\sigma)$ if and only if $\X$ is odd.  
\label{twr:OddElement} 
\end{theo}
Depending on the point of view it can be interpreted both as a positive
and a negative solution.

Taking into account the failure of Marcinkiewicz theorem this confirms 
the free analogue of the $\chi^2$-conjecture in the broad sense.

If we suppose in addition that the distribution is even, 
then the above theorem gives a positive answer to the free analogue of
Ruben's first theorem \cite{Ruben:1974}.

\begin{prop}
 Let $\X_1, \X_2,\dots, \X_n $ be free identically distributed random variables with finite non-zero variance $\sigma^2$, and assume that the 
distribution of $\X_1$ is symmetric.
 Then $\Q_n$
is distributed as $\chi^2(n-1,\sigma)$ if and only if $\X_1$ has
Wigner semicircular law. 
\label{twr:1} 

\end{prop}

On the other hand, the free analogue of Ruben's second theorem
\cite{Ruben:1975} (see also \cite{Bondesson:1977}) does not hold:

\begin{prop}
Let $\X_1, \X_2,\dots$ denote free independently and identically distributed random variables with finite non-zero variance $\sigma^2$. Let $m, n$ denote 
distinct integers not less than 2.
 Then for $\Q_n/\sigma^2$ and $\Q_m/\sigma^2$
to be distributed as $\chi^2(n-1)$ and $\chi^2(m-1)$, respectively,
it is not necessary that $\X_1$ is semicircular.  
\label{twr:2}
\end{prop}

\begin{proof}[Proof of  Proposition~\ref{twr:1}]
 If  $\Q_n$
is distributed as $\chi^2(n-1,\sigma)$  then $\X_1$ is odd, but taking into account that $\X_1$ is symmetric we have that its odd central moments vanish, and therefore its odd cumulants higher than the first vanish, so $\X_1$ has Wigner semicircular distribution.
\\
\textit{Proof of  Proposition~\ref{twr:2}.}
  Assume that $K_1(\X_1) = 0$, $K_2(\X_1) = \sigma^2$, $K_3(\X_1) = \epsilon$ and $K_i(\X_1)=0$ for $i\geq 4$ where $\epsilon$ is small enough. By Theorem   \ref{twr:OddElement} we see that $\Q_n$ and $\Q_m$ have  $\chi^2(n-1,\sigma)$ and $\chi^2(m-1,\sigma)$ distribution respectively.

\end{proof}

\begin{Rem}
  In this paper we assume that the involved random variables are bounded, 
  that is $\X_i \in \A$, as was common practice for a long time.
  Recently however unbounded random variables, i.e., operators affiliated with
  the von Neumann algebra in question, came into the focus of research.
  This happened in particular in connection with certain characterization
  problems, see, e.g.,
  \cite{ChistyakovGoetzeLehner:2011,EjsmontFranzSzpojankowski:2015,
    ChistyakovGoetzeLehner:2016}.
  It follows from the following result Chistyakov and Goetze
  that for the characterization problems pertinent to the present paper
  the question of boundedness is unessential.
\end{Rem}

\begin{lemm}[{  \cite[Lemma~3.10]{ChistyakovGoetze:2011}}]
 \label{lemm:bounded}
 Assume that $\mu=\mu_1\boxplus\mu_2$, where $\mu$ has compact support. 
 Then $\mu_1$ and $\mu_2$ have compact support as well. 
\end{lemm}

In terms of operators this means that
if $\X,\Y \in \widetilde{\mathcal{A}}_{sa}$ are free random variables
affiliated with $\mathcal{A}$ and 
such that  $\X+\Y\in \A_{sa}$, i.e., $\X+\Y$ is bounded,
then   $\X,\Y\in \A_{sa}$.
Now we will show that Theorem \ref{twr:OddElement} 
is true under  weaker conditions.
\begin{cor} 
 Let $\X_1, \X_2,\dots, \X_n \in \widetilde{\mathcal{A}}_{sa}$ 
 be selfadjoint free random variables and assume
 $\Q_n=\frac{1}{n}\sum_{i<j}(X_i-X_j)^2$ is bounded. Then all $X_i$ are bounded.
\end{cor}
\begin{proof}
Since $\X_i$ are self-adjoint and
$n\Q_n=\sum_{1\leq i \leq j \leq n} (\X_i-\X_j)^2$ is bounded,
it follows that $(\X_i-\X_j)^2$ is bounded
and hence also  $\X_i-\X_j$. 
By Lemma  \ref{lemm:bounded} we deduce that all $\X_i$ are bounded.
\end{proof}

 The proof of Ruben's theorem \cite{Ruben:1974} heavily relies  on  the symmetry
of random variables. Is it possible to drop the hypothesis that the
random variables are symmetric? Golikova   and Kruglov  \cite{GolikovaKruglov:2015}  
give a partial answer to this question -- instead of symmetry of $\X_1$ they assume infinite divisibility. The following is a free version of their result
which characterizes the classical normal law by the sample variance. 
We drop the assumption that $\X_i$ have  the same distribution, because  with this assumption the result follows directly from Theorem \ref{twr:OddElement} (we cannot use the Bercovici-Pata bijection  to prove it because it does not map classical chi-square to free chi-square distributions). 

\begin{prop}
Let $\X_1, \X_2,\dots \X_n$ denote free independent, freely infinitely divisible random variables with mean  $\tau(\X_1)=\tau(\X_2)= \dots =\tau(\X_n)$ and $\Var(\X_1)=\Var(\X_2)= \dots =\Var(\X_n)=1$. 
 Then if $\Q_n$ 
is distributed as free $\chi^2(n-1)$ if and only if  $\X_1, . . . , \X_n$ are
identically distributed Wigner semicircular random variables. 
\label{twr:kruglow}
\end{prop}

We conclude with a free version of a the following result of
  Kagan and Letac \cite{KaganLetac:1999}:
Fix an integer $n\geq 3$ and let $\X_1, \X_2,\dots, \X_n$ 
be independent  identically distributed  random variables.  
Consider the linear subspace $E=\mathbf{1}^\perp$ of Euclidean space $\R^n$ , i.e.,
the hyperplane $E=\{(a_1, a_2,\cdots,a_n ) : a_1+a_2+\cdots+a_n =0\}$. 
Then the following characterizations hold:
\begin{enumerate}[(i)]
\item  If the distribution of the $E$-valued random variable
 $$V=(\X_1-\overline{\X},\dots,\X_n-\overline{\X})$$
is invariant under all rotations of the Euclidean space $E$, then the $\X_i$'s
are normally distributed.
\item If the distribution of the random variable 
$$\sum_{i=1}^n(\X_i-\overline{\X}+a_i)^2$$
does not change as the real parameters $a_i$ vary on a sphere (i.e.,
the euclidean length $\norm{a}^2=a_1^2+a^2_2+\cdots+a_n^2 $ remains constant), 
then the $\X_i$'s are normally distributed.
\end{enumerate}
A key ingredient of the proof of these classical results is played by
\emph{Marcinkiewicz' theorem}.
As we discussed above, Marcinkiewicz' theorem has no analogue in free probability
and we will use different methods to prove the following free version of
 \cite{KaganLetac:1999}.
 This method also works in classical probability if  we assume that all moment exists. 
\begin{prop} 
Let $n$ be a fixed integer $n\geq 3$. Let $\X_1, \X_2,\dots, \X_n$ be free identically distributed  random variables.   
\begin{enumerate}
 \item  If for all $a \in E\subset \R^n$ 
  the distribution of the random variable 
  $$\sum_{i=1}^n(\X_i-\overline{\X}+a_i)^2$$
  depends only on $\norm{a}^2=a_1^2+a^2_2+\cdots+a_n^2 $, 
  then the $\X_i$'s obey the semicircle law.
\item If the distribution of the $E$-valued random variable
 $$V=(\X_1-\overline{\X},\dots,\X_n-\overline{\X})$$
is invariant under all rotations of the Euclidean space $E$, then the $\X_i$'s
obey the semicircle law.
\end{enumerate}

\label{twr:OdpowidenikKaganLetac} 
\end{prop}

\subsection{Noncrossing Partitions}
Let $S$ be finite subset of  $\N$.
A partition of $S$ is a set of mutually disjoint subsets
(also called \emph{blocks}) $B_1,B_2,\dots,B_k\subseteq S$ 
whose union is $S$. Any partition $\pi$ defines an equivalence relation on $S$,
denotes $\sim_\pi$, such that the equivalence classes are the blocks $\pi$. 
That is, $i\sim_\pi j$ if $i$ and $j$ belong to the same block of $\pi$.
A partition $\pi$ is called \emph{noncrossing} 
if different blocks do not interlace, i.e., there is no quadruple of
elements $i<j<k<l$ such that $i\sim_\pi k$ and $j\sim_\pi l$ 
but $i\not\sim_\pi j$. 

The set of non-crossing partitions of $S$ is denoted by $\NC(S)$,
in the case where $S=[n]:=\{1, \dots , n\}$ we write
$\NC(n):=\NC([n])$. 
$\NC(n)$ is a poset under refinement order, where we say $\pi\leq \rho$ if
every block of $\pi$ is contained in a block of $\rho$.
It turns out that $\NC(n)$ is in fact a lattice, 
see \cite[Lecture~9]{NicaSpeicher:2006}.

The maximal element of $\NC(n)$ under this order
is denoted by  $\hat{1}_{n}$. It  is the partition consisting 
of only one block. On the other hand the minimal element $\hat{0}_n$ 
is the unique
partition where every block is a singleton.

Sometimes it is convenient to visualize partitions as diagrams, for example
$\hat{1}_n=
\begin{picture}(26,6.5)(1,0)
  \put(2,0){\line(0,1){7.5}}
  \put(8,0){\line(0,1){7.5}}
  \put(10,0){$\cdots$}
  \put(26,0){\line(0,1){7.5}}
  \put(2,7.5){\line(1,0){24}}
\end{picture}
$
and $\hat{0}_n=
\begin{picture}(26,3.5)(1,0)
  \put(2,0){\line(0,1){4.5}}
  \put(8,0){\line(0,1){4.5}}
\put(10,0){$\cdots$}
  \put(26,0){\line(0,1){4.5}}
  \put(2,4.5){\line(1,0){0}}
  \put(8,4.5){\line(1,0){0}}
  \put(14,4.5){\line(1,0){0}}
  \put(20,4.5){\line(1,0){0}}
  \put(26,4.5){\line(1,0){0}}
\end{picture}$.

\subsection{Some Special Notations}
We will be concerned with certain special classes of noncrossing partitions.
If $n$ is even we denote by  $\NCeven(n)$ 
the subset of even noncrossing partitions,
where we say that a partition is even if all its blocks have even cardinality.
Even more specific we denote by $\NC_2(n)$ is the set of all noncrossing pair
partitions, i.e., partitions where every block has size $2$.

Two specific minimal pair partitions will play a particularly important role,
namely
$\onetwo{r}=
\begin{picture}(44,3.5)(1,0)
  \put(2,0){\line(0,1){4.5}}
  \put(8,0){\line(0,1){4.5}}
  \put(14,0){\line(0,1){4.5}}
  \put(20,0){\line(0,1){4.5}}
\put(21.5,0){$\cdots$}
  \put(38,0){\line(0,1){4.5}}
  \put(44,0){\line(0,1){4.5}}
  \put(2,4.5){\line(1,0){6}}
  \put(14,4.5){\line(1,0){6}}
  \put(38,4.5){\line(1,0){6}}
\end{picture}\in\NC(2r)$, which is a kind of blow up
of $\hat{1}_r$ and its shift $\pispecial{r}=
\begin{picture}(56,6.5)(1,0)
  \put(2,0){\line(0,1){7.5}}
  \put(8,0){\line(0,1){4.5}}
  \put(14,0){\line(0,1){4.5}}
  \put(20,0){\line(0,1){4.5}}
  \put(26,0){\line(0,1){4.5}}
  \put(44,0){\line(0,1){4.5}}
  \put(50,0){\line(0,1){4.5}}
  \put(56,0){\line(0,1){7.5}}
\put(28.0,0){$\cdots$}
  \put(8,4.5){\line(1,0){6}}
  \put(20,4.5){\line(1,0){6}}
  \put(44,4.5){\line(1,0){6}}
  \put(2,7.5){\line(1,0){54}}
\end{picture}\in\NC(2r)$.

In the proof of Theorem~\ref{twr:OddElement} we will use telescoping argument
and put a filtration on $\NC(n)$ by avoiding certain blocks. For this
purpose we introduce the following notation.

For a subset $B\subseteq \N$ let
$\NC^{B}(n):=\{\pi\in \NC(n): B\in\pi\}$, i.e., the collection
of noncrossing partitions which contain $B$ as a block.
On the other hand, 
for a family $B_1,B_2,\dots,B_m\subseteq \N$ of subsets 
let 
$\NC_{B_1,\dots,B_m}(S):=\{\pi\in \NC(S):\pi\cap\{B_1,\dots,B_m\}=\emptyset\}$, 
i.e., the collection of noncrossing
partition which do not contain any $B_i$ as a block.
Finally, combining the two notations we define
$\NC^{B}_{B_1,\dots,B_m}(n):=\NC^{B}(n)\cap\NC_{B_1,\dots,B_m}(n)$.

\subsection{Kreweras Complements}
Kreweras \cite{Kreweras:1972} discovered an interesting antiisomorphism
of the lattice $\NC(n)$, now called the \emph{Kreweras complementation map},
of which we will need two variants.
Given a noncrossing partition $\pi$ of $\{1,2,\dots,n\}$,
the \emph{left Kreweras complement} $\Krewl\pi$ is the maximal
noncrossing partition of the ordered set $\{\bar{1},\bar{2},\dots,\bar{n}\}$
such that $\pi\cup\Krewl\pi$ is a noncrossing partition of
the interlaced set
$\{\bar{1},1,\bar{2},2,\dots,\bar{n},n\}$.
Similarly, the \emph{right Kreweras complement} $\Krewr\pi$ is the maximal
noncrossing partition of the ordered set $\{\bar{1},\bar{2},\dots,\bar{n}\}$
such that $\pi\cup\Krewr\pi$ is a noncrossing partition of
the interlaced set
$\{1,\bar{1},2,\bar{2},\dots,n,\bar{n}\}$.
It is then clear that $\Krewr\circ\Krewl=\mathrm{id}$ and it can be shown
that 
\begin{equation}
\label{eq:cardKrew}
\abs{\Krewr\pi}=\abs{\Krewl\pi}=n+1-\abs\pi
.
\end{equation}

Finally we define the \emph{extended Kreweras complement} $\Krewlr\pi$
to be the maximal
noncrossing partition of the ordered set $\{\bar{0},\bar{1},\dots,\bar{n}\}$
such that $\pi\cup\Krewlr\pi$ is a noncrossing partition of
the interlaced set
$\{\bar{0},1,\bar{1},2,\bar{2},\dots,n,\bar{n}\}$.
The extended Kreweras complement is always irreducible,
i.e., $\bar{0}$ and $\bar{n}$ are in the same block of $\Krewlr\pi$. 
In fact it is obtained by joining $\bar{0}$ to the last block of $\Krewr\pi$, 
i.e., the block containing $\bar{n}$,
or by joining $n+1$ to the first block of $\Krewl\pi$.
The following observation is useful for recursive proofs involving 
the Kreweras complement(s). 
\begin{lemm}
  \label{lemm:kreweras}
  Let $\pi\in\NC(n)$ and $B=\{j_1,j_2,\dots,j_p=n\}$ be its last block.
  Let $\pi_1,\pi_2,\dots,\pi_p$ be the restrictions of $\pi$ to the
  maximal intervals of $\{1,2,\dots,n\}\setminus B$ as shown in
  the following picture:
  $$
\begin{picture}(112,13.0)(1,0)
\put(-1,0){$\pi_1$}
\put(16,0){\line(0,1){13.0}}
\put(23,0){$\pi_2$}
\put(40,0){\line(0,1){13.0}}
\put(47,0){$\pi_3$}
\put(64,0){\line(0,1){13.0}}
\put(69,0){$\dotsm$}
\put(69,10){$\dotsm$}
\put(88,0){\line(0,1){13.0}}
\put(95,0){$\pi_p$}
\put(112,0){\line(0,1){13.0}}
\put(4,7.0){\line(1,0){0}}
\put(28,7.0){\line(1,0){0}}
\put(52,7.0){\line(1,0){0}}
\put(76,7.0){\line(1,0){0}}
\put(100,7.0){\line(1,0){0}}
\put(16,13.0){\line(1,0){52}}
\put(84,13.0){\line(1,0){28}}
\end{picture}
  $$
  Then the left Kreweras complement of $\pi$ is the concatenation
  of the extended Kreweras complements of the subpartitions $\pi_j$:
  $$
  \Krewl\pi = \Krewlr{\pi_1}\,\Krewlr{\pi_2}\dotsm\Krewlr{\pi_p}.
  $$
\end{lemm}

\subsection{Free Cumulants}

Let $\IC \langle \X_{1},\dots ,\X_{n}   \rangle$ denote
the non-commutative ring of polynomials in variables $\X_{1},\dots
,\X_{n}\in\A $.
The free  cumulants are multilinear maps $K_r : \mathcal{A}^r\to\IC$ 
defined  implicitly by the relation (connecting them with mixed moments)
\begin{align}
\tau(\X_{1}\X_{2}\dots \X_{n}) = \sum_{\pi \in \NC(n)}K_{\pi}(\X_{1},\X_{2},\dots ,\X_{n}),\label{eq:DefinicjaKumulant}
\end{align}
where 
\begin{align}
K_{\pi}(\X_{1},\X_{2},\dots ,\X_{n}):=\Pi_{B \in \pi}K_{|B|}(\X_{i}:i \in B) \label{eq:DefinicjaProduktuKumulant}
\end{align}
and $\NC(n)$ is the set of all non-crossing partitions of $\{1, 2,\dots, n \}$
(see \cite{NicaSpeicher:2006}). Sometimes we will write $K_{r}(\X)=K_{r}(
\X,\dots ,\X )$. 

Free cumulants provide the most important technical tool to investigate
free random variables. 
This is due to the basic property of \emph{vanishing of mixed cumulants}. 
By this we mean the fact that 
$$
K_r(X_1,X_2,\dots,X_n)=0
$$
for any family of random variables $X_1,X_2,\dots,X_n$ which can be partitioned
into two free subsets.

For free sequences this can be reformulated as follows.
Let $(X_i)_{i\in\N}$ be a sequence of free random variables and
$h:[r]\to\N$ a map. We denote by $\ker h$ the set partition
which is induced by the equivalence relation 
$$
k\sim_{\ker h} l
\iff
h(k)=h(l)
.
$$
Similarly, for a multiindex $\underline{i} = i_1i_2\dots i_n$ we denote its
kernel
$\ker\underline{i}$ by the relation $k\sim l$ if $i_k=i_l$.

Using this notation, we have that
\begin{equation}
  \label{eq:kerh>=pi}
  K_\pi(X_{h(1)},X_{h(2)},\dots,X_{h(r)})=0
  \text{ unless $\ker h\geq \pi$.}
\end{equation}

Our main technical tool is the free version,
due to Krawczyk and Speicher~\cite{KrawczykSpeicher:2000} (see also   \cite[Theorem 11.12]{NicaSpeicher:2006}), of the classical formula of 
James/Leonov and Shiryaev \cite{James:1958,LeonovShiryaev:1959} which
expresses cumulants of products in terms of individual cumulants.
\begin{theo}
\label{thm:krawczyk}
Let
$r,n \in \N$ and $ i_1 < i_2 < \dots < i_r = n$ be given and let
$$\rho=\{(1,\dots,i_1),\dots,(i_{r-1}+1,\dots,i_r)\}\in \NC(n)$$ 
be the induced interval partition.
Consider now random variables $\X_1,\dots,\X_n\in\A$.
Then the free cumulant of the products can be expanded
as follows:
\begin{align} 
K_r(\X_1\dots \X_{i_1},\dots,\X_{i_{r-1}+1}\dots\X_n)=\sum_{\substack{\pi\in \NC(n) \\ \pi\vee\rho=\hat{1}_{n}} }K_\pi({\X_1,\dots,\X_n}). 
\end{align} \label{twr:produktargumentow}
\end{theo}

In the special case of products of pairs of free elements this yields
the following formula for multiplicative free convolution.
\begin{theo}[{\cite[Theorem~14.4]{NicaSpeicher:2006}}]
  \label{thm:multfreeconvolution}
  Let $\{X_1,X_2,\dots,X_r\}$ and  $\{Y_1,Y_2,\dots,Y_r\}$ be two
  mutually free sets of random variables, then
  $$
  K_r(X_1Y_1,X_2Y_2,\dots,X_rY_r)=\sum_{\pi\in\NC(r)} K_\pi(X_1,X_2,\dots,X_r)\,
  K_{\Krewr{\pi}}(Y_1,Y_2,\dots,Y_r)
  $$
\end{theo}
This motivates the following definition.
\begin{defi}[{\cite[Ch.~17]{NicaSpeicher:2006}}]
  \label{def:boxedconvolution}
  Let $(a_n)_{n\geq0}$ a sequence and  $\pi\in\NC(n)$ a noncrossing partition.
  As in \eqref{eq:DefinicjaProduktuKumulant} we denote by
  $a_\pi$ the product 
  $$
  a_\pi=\prod_{B\in\pi} a_{\abs{B}}
  .
  $$ 
  Given two sequences  $(a_n)_{n\geq0}$ and $(b_n)_{n\geq0}$ with
  respective generating formal power series
  $f(z)=\sum_{n=1}^\infty a_nz^n$ and  $g(z)=\sum_{n=1}^\infty b_nz^n$ 
  we define their \emph{boxed convolution} as the formal series
  $f\boxstar g(z)=\sum_{n=1}^\infty c_nz^n$ with coefficients
  $$
  c_n = \sum_{\pi\in\NC(n)} a_\pi b_{\Krewr{\pi}}
      = \sum_{\pi\in\NC(n)} a_{\Krewl\pi} b_{\pi}
  .
  $$
\end{defi}
As examples, consider the univariate case
of Theorem~\ref{thm:multfreeconvolution}, which can be rewritten
as $R_{XY}(z) = R_X\boxstar R_Y(z)$,
or \cite[Proposition~11.25]{NicaSpeicher:2006}, which states
that the $R$-transform of an even element $X$  can be written as
$$
R_{X^2}(z) = \alpha\boxstar \zeta(z)
$$  
where
$\alpha(z) = \sum_{n=1}^\infty K_{2n}(X)z^n$ is the determining series of $X$
and $\zeta(z)=\sum_{n=1}^\infty z^n$ is the so-called \emph{Zeta-series}.
Combinatorially this means
\begin{equation}
  \label{eq:KX2}
  K_r(X^2) = \sum_{\pi\in\NC(r)}\alpha_\pi = \sum_{\pi\in\NC(r)}\prod_{B\in\pi} K_{2\abs{B}}(X).
\end{equation}

\noindent

The next result follows from  \cite[Proposition 2.2]{HiwatashiKurodaNagisaYoshida:1999}; see Corollary~\ref{cor:KrQn} below for a generalization.
\begin{prop}
\label{twr:Lehner:2003}
Let $\X_1,\X_2,\dots,\X_n$ be  free identically distributed Wigner semicircular random variables with mean zero and variance $\sigma^2$. Then the cumulants of $\Q_n$ are given as follows:
 \begin{align} 
 K_r(\Q_n) &=(n-1)\,\sigma^{2r}.
\end{align}
\end{prop}

The following lemma connects Theorem~\ref{thm:krawczyk}
with Definition~\ref{def:boxedconvolution} and is the key to the main result.
Its proof is contained in the proof of Proposition~11.25 in
 the book \cite{NicaSpeicher:2006}.
\begin{lemm} 
  \label{lemm:lematoparzystych}
  Let $r\in\mathbb{N}$ and $\pi \in \NCeven(2r)$,
  then $\pi \vee \onetwo{r}=\hat{1}_{2r}$ if and only if $\pi\geq \pispecial{r}$,
  i.e., $1$ and $2r$ lie in the same block of $\pi$ and elements $2i$ and $2i+1$
  also lie in the same block of $\pi$ for $i\in[r-1]$.  
  Consequently
  $$
  \{\pi : \pi \vee \onetwo{r}=\hat{1}_{2r}\}\cap \NCeven(2r)
  = [\pispecial{r}, \hat{1}_{2r}  ] ,
  $$
  is a lattice isomorphic to $\NC(r)$.
\end{lemm}
\begin{cor} \label{corr:dwupartycje}
There is
only one non-crossing pair partition $\pi$ such that  $\pi \vee
\onetwo{r}=\hat{1}_{2r} $, 
namely $\pispecial{r}=
\begin{picture}(56,6.5)(1,0)
  \put(2,0){\line(0,1){7.5}}
  \put(8,0){\line(0,1){4.5}}
  \put(14,0){\line(0,1){4.5}}
  \put(20,0){\line(0,1){4.5}}
  \put(26,0){\line(0,1){4.5}}
  \put(44,0){\line(0,1){4.5}}
  \put(50,0){\line(0,1){4.5}}
  \put(56,0){\line(0,1){7.5}}
\put(28.0,0){$\cdots$}
  \put(8,4.5){\line(1,0){6}}
  \put(20,4.5){\line(1,0){6}}
  \put(44,4.5){\line(1,0){6}}
  \put(2,7.5){\line(1,0){54}}
\end{picture}=\{(1,2r),(2,3),\dots,(2r-2,2r-1)\}$. 
\end{cor}
\begin{defi}
Let $B_1,\dots,B_m $ be subsets of $\N$
and random variables $\X_1,\dots,\X_n\in\A,$ be given.
Then for an interval partition $\rho=\{(1,\dots,i_1),(i_1+1,\dots,i_2),\dots,
(i_{r-1}+1,\dots,i_r)\}$ 
we define the partial cumulant functional
\begin{align} 
K^\rho_{B_1,\dots,B_m}(\X_1,\X_2,\dots,\X_n)=\sum_{\substack{\pi\in \NC_{B_1,\dots,B_m} (n)\\ \pi\vee\rho=\hat{1}_{n}} }K_\pi({\X_1,\dots,\X_n}),
\end{align} 
Usually we will abuse notation and abbreviate this expression as
\begin{align} 
K^r_{B_1,\dots,B_m}(\X_1\dots \X_{i_1},\dots,\X_{i_{r-1}+1}\dots\X_{i_r})=\sum_{\substack{\pi\in \NC_{B_1,\dots,B_m} (n)\\ \pi\vee \rho=\hat{1}_{n}} }K_\pi({\X_1,\dots,\X_n}).
\end{align} 

\end{defi}

\begin{lemm}
  \label{lem:centeredstatistic}
   Let $P=P(X_1,X_2,\dots,X_n)$ be a polynomial
    of degree at most two in  noncommuting variables
    $X_1,X_2,\dots,X_n$. Then $\tau(P(\X_1,\X_2,\dots,\X_n))=0$ 
    for every i.i.d.\ free family $X_i$ 
    if and only if 
    $$
    \sum_{\sigma\in\SG_n} P(X_{\sigma(1)},X_{\sigma(2)},\dots,X_{\sigma(n)})=0
    .
    $$
\end{lemm}
\begin{proof}
  Clearly by symmetry the second condition is stronger than the first condition.

  In order to show that it is also necessary, we first note that
  by a simple scaling argument we may assume without loss of generality
  that the polynomial in consideration is homogeneous.
  Clearly such a polynomial cannot have a constant term and
  we start with a linear polynomial $P=\sum_{i=1}^n \alpha_i X_i$.
  By evaluating a distribution with nonzero first moment it follows
  that $\sum_{i=1}^n \alpha_i=0$. But then we have
  $$
  \sum_{\sigma\in\SG_n}P(X_{\sigma(1)},X_{\sigma(2)},\dots,X_{\sigma(n)})
  =\sum_{i=1}^n\alpha_i(n-1)!\sum_{k=1}^n X_k=0
  .
  $$
  Let us now turn to a homogeneous polynomial of second order
  $$
  P = \sum_{i,j=1}^n \alpha_{ij}X_iX_j
  .
  $$
  Evaluating at a distribution with first moment $\mu_1$ and second moment $\mu_2$
  we obtain
  $$
  \sum_{i\ne j}\alpha_{ij}\mu_1^2 +   \sum_{i}\alpha_{ii}\mu_2 = 0
  $$
  and it follows that
  $$
  \sum_{i\ne j}\alpha_{ij}
  =
  \sum_i \alpha_{ii}
  =0
  .
  $$
  Now consider the symmetrization
  \begin{align*}
  \sum_{\sigma\in\SG_n}P(X_{\sigma(1)},X_{\sigma(2)},\dots,X_{\sigma(n)})
  &= \sum_{i,j}\sum_\sigma X_{\sigma(i)}X_{\sigma(j)} \\
  &= \sum_{i\ne j}\alpha_{ij}(n-2)!\sum_{k\ne l}  X_kX_l
     + \sum_i \alpha_{ii} (n-1)!\sum_k X_k^2
     \\
  &=0
  .
  \end{align*}
\end{proof}
\begin{Rem}
  \begin{enumerate}
   \item 
    Our typical example of a centered linear statistic will be 
    $\X_{i}-\overline{\X}$.
   \item 
  The example $P=X_1X_2X_1-X_1^2X_2$ shows that in the present formulation
  the lemma cannot be extended beyond degree $2$.
  \end{enumerate}
\end{Rem}

In the following a polynomial $P(X_1,X_2,\dots,X_n)$ in noncommuting variables
is called \emph{symmetric} if it is invariant under permutations, i.e.,
 $P(X_{\sigma(1)},X_{\sigma(2)},\dots,X_{\sigma(n)})=P(X_1,X_2,\dots,X_n)$ for any
permutation $\sigma\in\SG_n$. For a linear form $L=\sum_{i=1}^n \alpha_i X_i$ 
we denote the permuted form by $L_\sigma=\sum_{i=1}^n \alpha_iX_{\sigma(i)}$.

\begin{lemm}
\label{lemm:sumyobcietelematgeneral}
 Let $\X_1, \X_2,\dots, \X_n \in \A_{sa}$ be free identically distributed random variables, $L=\sum_{i=1}^n \alpha_i \X_i$ a linear form such that $\tau(L)=0$
and $P_j=P_j(X_1,X_2,\dots,X_n)$ symmetric polynomials for $j\in\{1,2,\dots,r\}\setminus\{k\}$.
Then
$$K^r_{B_1,\dots,B_m}(P_1,\dots,P_{k-1},L,P_{k+1},\dots,P_r)=0.$$
\end{lemm}
\begin{proof}
 Let us first observe that for $i\neq j$ we have 
$$K^r_{B_1,\dots,B_m}(P_1,\dots,P_{k-1},\X_{i},P_{k+1},\dots,P_r)=K^r_{B_1,\dots,B_m}(P_1,\dots,P_{k-1},\X_{j},P_{k+1},\dots,P_r),$$ 
for all  $i,j\in[n],$ and $r\geq 1$. This follows from the argument that $\X_i$ are free i.i.d. and $P_j$ are symmetric polynomials in the $n$ variables $\X_1,\dots,\X_n.$
It follows by multilinearity that
$$K^r_{B_1,\dots,B_m}(P_1,\dots,P_{k-1},L,P_{k+1},\dots,P_r)=K^r_{B_1,\dots,B_m}(P_1,\dots,P_{k-1},L_\sigma,P_{k+1},\dots,P_r),$$ for every permutation $\sigma\in\SG_n$
and taking the average, we have
$$
K^r_{B_1,\dots,B_m}(P_1,\dots,P_{k-1},L,P_{k+1},\dots,P_r)=
\frac{1}{n!}\sum_{\sigma\in\SG_n}
K^r_{B_1,\dots,B_m}(P_1,\dots,P_{k-1},L_\sigma,P_{k+1},\dots,P_r)
=0,
$$ 
again by multilinearity and taking into account Lemma~\ref{lem:centeredstatistic}.

\end{proof}

\begin{cor}
  \label{lemm:sumyobcietelemat}
   Let $\X_1, \X_2,\dots, \X_n \in \A_{sa}$ be free identically distributed random variables then 
$$K^r_{B_1,\dots,B_m}(\Q_n,\dots,\Q_n,\X_{i}-\overline{\X},\Q_n,\dots,\Q_n)=0.$$
\end{cor}

\section{Proof of the main theorem} 
Continuing Lemma~\ref{lem:centeredstatistic} we 
establish a curious cancellation result for symmetrized squares of centered
linear statistics. A similar phenomenon was observed by
Nica and Speicher  \cite[Theorem 1.2]{NicaSpeicher:1998} in the case
of the free commutator. We postpone the investigation of a possible common
pattern between these phenomena to future work.
\begin{lemm}
  \label{lem:oddcancellation}
  Let $\X_1, \X_2,\dots, \X_n$ be free identically distributed copies
  of a random variable $X$ 
  and
  $L=\sum_{i=1}^n\alpha_i X_i$ a linear form with $\tau(L)=0$.
  Then the distribution of the quadratic statistic
  $$
  P=\sum_{\sigma\in \SG_n} L_\sigma^2,
  $$
  does not depend on the odd cumulants of $\X$.
\end{lemm}

\begin{proof}
We show by induction that the cumulants of $P$ can be expressed in 
terms of the even cumulants of $X$.
First we apply the product formula of Theorem~\ref{thm:krawczyk} and obtain
  \begin{align*}
    K_r(P) &= \sum_{\sigma_1,\dots,\sigma_r\in\SG_n}
               K_r(L_{\sigma_1}^2,L_{\sigma_2}^2,\dots,L_{\sigma_r}^2)\\
           &= \sum_{\sigma_1,\dots,\sigma_r\in\SG_n}
               \sum_{\substack{\pi\in \NC(2r)\\ \pi\vee \onetwo{r}=\hat{1}_{2r}}}
               K_\pi(L_{\sigma_1},L_{\sigma_1},
                    L_{\sigma_2},L_{\sigma_2},
                    \dots,
                    L_{\sigma_r},L_{\sigma_r})\\
           &= \sum_{\substack{\pi\in \NC(2r)\\ \pi\vee \onetwo{r}=\hat{1}_{2r}}}
              \tilde{K}_\pi(L)
  \end{align*}
  where for $\pi\in\NC(2r)$ we write
  \begin{equation}
    \label{eq:Ktilde}
  \tilde{K}_\pi(L)
  =\sum_{\sigma_1,\sigma_2,\dots,\sigma_r\in\SG_n} 
    K_\pi(L_{\sigma_1},L_{\sigma_1},
          L_{\sigma_2},L_{\sigma_2},
          \dots,
          L_{\sigma_r},L_{\sigma_r}).
  \end{equation}
  We claim that in this decomposition
  the contributions of non-even partitions cancel each other.
  To see this, we proceed by induction and use 
  Lemma~\ref{lemm:sumyobcietelematgeneral}.
  Let $B_1,B_2,\dots,B_m$ be an enumeration of all odd subsets of $[2r]$,
  then we can split off the sum \eqref{eq:Ktilde} the even part
  and decompose the rest in a ``telescope'' fashion
  as 
  \begin{equation}
    \label{eq:telescope}
    K_r(P) = 
    \sum_{\substack{\pi\in \NCeven(2r)\\ \pi\vee\onetwo{r}=\hat{1}_{2r}}}
    \tilde{K}_\pi(L)
    +\sum_{k=1}^m
    \sum_{\substack{\pi\in \NC^{B_k}_{B_1,B_2,\dots,B_{k-1}}(2r)\\ \pi\vee\onetwo{r}=\hat{1}_{2r}}}
    \tilde{K}_\pi(L)
    .
  \end{equation}
   The last formula is obtained using
the following decomposition 
$$\NC(2r)\setminus \NCeven(2r)=NC^{B_1}(2r) \cup NC^{B_2}_{B_1}(2r)\cup NC^{B_3}_{B_1,B_2}(2r) \cup\dots \cup NC^{B_m}_{B_1,\dots,B_{m-1}}(2r).$$ 
Directly from the definition we have $NC^{B_i}_{B_1,\dots,B_{i-1}}(2r)\cap
NC^{B_j}_{B_1,\dots,B_{j-1}}(2r)=\emptyset,$ for $i\neq j$.   We  will show the
inclusion
$$\NC(2r)\setminus \NCeven(2r)\subseteq NC^{B_1}(2r) \cup NC^{B_2}_{B_1}(2r)\cup NC^{B_3}_{B_1,B_2}(2r) \cup\dots \cup NC^{B_m}_{B_1,\dots,B_{m-1}}(2r)$$ 
only because the converse inclusion is obvious.  
Given $\pi\in \NC(2r)\setminus \NCeven(2r)$, let $k$ be the smallest index
such that $B_k\in\pi$ then $\pi\in NC^{B_{k}}_{B_1,\dots,B_{k-1}}(2r)$ for the blocks
$B_1,B_2, \dots,B_{k-1}$ do not appear in $\pi$.

  It remains to show that each individual sum
  \begin{equation}
    \label{eq:sumNCBk}
    \sum_{\substack{\pi\in \NC^{B_k}_{B_1,B_2,\dots,B_{k-1}}(2r)\\ \pi\vee\onetwo{r}=\hat{1}_{2r}}}
    \tilde{K}_\pi(L)
  \end{equation}
  vanishes. 
  Every $\pi$ in this sum contains the odd block $B_k$ and splits
  the complement $[2r]\setminus B_k$ into intervals
  $I_1,I_2,\dots,I_l$, interpreted in a circular manner,
  see Fig.~\ref{fig:oddinterval}. Then
  at least one of these intervals must be odd. 
  To simplify the discussion we may assume that either $I_1$ 
  is odd and $2r\in B_k$ or $I_l$ is odd and $1\in B_k$;
  this may always be achieved by applying an even rotation, which does
  not change the values of the cumulants because of traciality.
  We are now in one of the situations depicted in Fig.~\ref{fig:oddinterval}.
  \begin{figure}
      \begin{picture}(140,30)(20,0)
  \put(8,0){\line(0,1){10.5}}
  \put(44,0){\line(0,1){10.5}}
  \put(128,0){\line(0,1){10.5}}
  \put(8,10.5){\line(1,0){68}}
  \put(81,8.5){\tiny $\cdots$}
  
  \put(96,10.5){\line(1,0){32}}
  \put(6,-8){\tiny $j_1$}
  \put(41,-8){\tiny $j_2$}
  \put(126,-8){\tiny $j_l$}
  
  \put(-13,3){\tiny $\cdots$}
  \put(-10,-5){\tiny $I_1$}
  
  \put(20,3){\tiny $\cdots$}
  \put(22,-5){\tiny $I_2$}
  
  \put(55,3){\tiny $\cdots$}
  \put(57,-5){\tiny $I_3$}
  
  \put(81,-5){\tiny $\cdots$}
  
  \put(109,3){\tiny $\cdots$}
  \put(111,-5){\tiny $I_l$}

\end{picture}
\begin{picture}(86,9.5)(1,0)
  \put(8,0){\line(0,1){10.5}}
  \put(44,0){\line(0,1){10.5}}
  \put(128,0){\line(0,1){10.5}}
  \put(8,10.5){\line(1,0){68}}
  \put(81,8.5){\tiny $\cdots$}
  
  \put(96,10.5){\line(1,0){32}}
  \put(6,-8){\tiny $j_1$}
  \put(41,-8){\tiny $j_2$}
  \put(126,-8){\tiny $j_l$}
  
  \put(22,-5){\tiny $I_1$}
  
  \put(57,-5){\tiny $I_2$}
  
  \put(81,-5){\tiny $\cdots$}
  
  \put(107,-5){\tiny $I_{l-1}$}

  \put(142,-5){\tiny $I_l$}

\end{picture}
    \caption{Two types of partitions with an odd block}
    \label{fig:oddinterval}
  \end{figure}
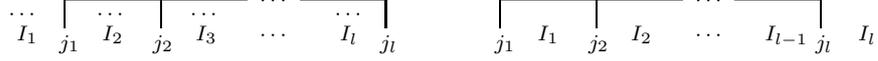
We concentrate on the first case, i.e., $I_1=\{1,2,\dots,j_1-1\}$ 
and $B_k=\{j_1,j_2,\dots,j_l\}$ where
$j_1$ is even, and $j_l=2r$.
Observe that every partition
$\pi\in \NC^{B_k}_{B_1,B_2,\dots,B_{k-1}}(2r)$ is the concatenation
of some noncrossing partition $\pi'\in
\NC_{B_1,B_2,\dots,B_{k-1}}(j_1-1)$
and $\pi''\in \NC^{B_k}_{B_1,B_2,\dots,B_{k-1}}(\{j_1,\dots,2r\})$.
Now $\pi\vee \onetwo{r}=\hat{1}_{2r}$ if and only if
$\pi'\vee 
\begin{picture}(62,3.5)(1,0)
  \put(2,0){\line(0,1){4.5}}
  \put(8,0){\line(0,1){4.5}}
  \put(14,0){\line(0,1){4.5}}
  \put(20,0){\line(0,1){4.5}}
\put(28.5,0){$\cdots$}
  \put(50,0){\line(0,1){4.5}}
  \put(56,0){\line(0,1){4.5}}
  \put(62,0){\line(0,1){4.5}}
  \put(2,4.5){\line(1,0){6}}
  \put(14,4.5){\line(1,0){6}}
  \put(50,4.5){\line(1,0){6}}
  \put(62,4.5){\line(1,0){0}}
\end{picture}
=\hat{1}_{j_1-1}
$
and 
$\pi''\vee 
\begin{picture}(62,3.5)(1,0)
  \put(2,0){\line(0,1){4.5}}
  \put(8,0){\line(0,1){4.5}}
  \put(14,0){\line(0,1){4.5}}
  \put(20,0){\line(0,1){4.5}}
  \put(26,0){\line(0,1){4.5}}
\put(33.5,0){$\cdots$}
  \put(56,0){\line(0,1){4.5}}
  \put(62,0){\line(0,1){4.5}}
  \put(2,4.5){\line(1,0){0}}
  \put(8,4.5){\line(1,0){6}}
  \put(20,4.5){\line(1,0){6}}
  \put(56,4.5){\line(1,0){6}}
\end{picture}
=\hat{1}_{\{j_1,\dots,2r\}}$.
Thus we may unfold \eqref{eq:Ktilde} and factor the sum \eqref{eq:sumNCBk} 
 to obtain
\begin{multline*}
  \sum_{\substack{\pi\in \NC^{B_k}_{B_1,B_2,\dots,B_{k-1}}(2r)\\ \pi\vee\onetwo{r}=\hat{1}_{2r}}}
  \tilde{K}_\pi(L)
  = \sum_{\substack{\pi'\in \NC_{B_1,B_2,\dots,B_{k-1}}(j_1-1)
\\ \pi'\vee 
\,
{
\setlength{\unitlength}{0.7\unitlength}
\tiny
\begin{picture}(62,3.5)(1,0)
  \put(2,0){\line(0,1){4.5}}
  \put(8,0){\line(0,1){4.5}}
  \put(14,0){\line(0,1){4.5}}
  \put(20,0){\line(0,1){4.5}}
\put(28.5,0){$\cdots$}
  \put(50,0){\line(0,1){4.5}}
  \put(56,0){\line(0,1){4.5}}
  \put(62,0){\line(0,1){4.5}}
  \put(2,4.5){\line(1,0){6}}
  \put(14,4.5){\line(1,0){6}}
  \put(50,4.5){\line(1,0){6}}
  \put(62,4.5){\line(1,0){0}}
\end{picture}
}
\,
=\hat{1}_{j_1-1}}}
    \tilde{K}_{\pi'}(L)
   \sum_{\substack{\pi''\in \NC^{B_k}_{B_1,B_2,\dots,B_{k-1}}(\{j_1,\dots,2r\}) \\
\\ \pi''\vee 
\,
{
\setlength{\unitlength}{0.7\unitlength}
\tiny
\begin{picture}(62,3.5)(1,0)
  \put(2,0){\line(0,1){4.5}}
  \put(8,0){\line(0,1){4.5}}
  \put(14,0){\line(0,1){4.5}}
  \put(20,0){\line(0,1){4.5}}
  \put(26,0){\line(0,1){4.5}}
\put(33.5,0){$\cdots$}
  \put(56,0){\line(0,1){4.5}}
  \put(62,0){\line(0,1){4.5}}
  \put(2,4.5){\line(1,0){0}}
  \put(8,4.5){\line(1,0){6}}
  \put(20,4.5){\line(1,0){6}}
  \put(56,4.5){\line(1,0){6}}
\end{picture}
}
\,
=\hat{1}_{\{j_1,\dots,2r\}}}}
    \tilde{K}_{\pi''}(L)
\\
\begin{aligned}
&  = \sum_{\sigma_1,\sigma_2,\dots,\sigma_r\in\SG_n}\sum_{\substack{\pi'\in \NC_{B_1,B_2,\dots,B_{k-1}}(j_1-1)
\\ \pi'\vee 
\,
{
\setlength{\unitlength}{0.7\unitlength}
\tiny
\begin{picture}(62,3.5)(1,0)
  \put(2,0){\line(0,1){4.5}}
  \put(8,0){\line(0,1){4.5}}
  \put(14,0){\line(0,1){4.5}}
  \put(20,0){\line(0,1){4.5}}
\put(28.5,0){$\cdots$}
  \put(50,0){\line(0,1){4.5}}
  \put(56,0){\line(0,1){4.5}}
  \put(62,0){\line(0,1){4.5}}
  \put(2,4.5){\line(1,0){6}}
  \put(14,4.5){\line(1,0){6}}
  \put(50,4.5){\line(1,0){6}}
  \put(62,4.5){\line(1,0){0}}
\end{picture}
}
\,
=\hat{1}_{j_1-1}}}
    K_{\pi'}(L_{\sigma_1}, L_{\sigma_1},\dots,L_{\sigma_{j_1'-1}},L_{\sigma_{j_1'-1}},L_{\sigma_{j_1'}})
\\
&\phantom{=xxxxxxxxxx}\times
   \sum_{\substack{\pi''\in \NC^{B_k}_{B_1,B_2,\dots,B_{k-1}}(\{j_1,\dots,2r\}) \\
\\ \pi''\vee 
\,
{
\setlength{\unitlength}{0.7\unitlength}
\tiny
\begin{picture}(62,3.5)(1,0)
  \put(2,0){\line(0,1){4.5}}
  \put(8,0){\line(0,1){4.5}}
  \put(14,0){\line(0,1){4.5}}
  \put(20,0){\line(0,1){4.5}}
  \put(26,0){\line(0,1){4.5}}
\put(33.5,0){$\cdots$}
  \put(56,0){\line(0,1){4.5}}
  \put(62,0){\line(0,1){4.5}}
  \put(2,4.5){\line(1,0){0}}
  \put(8,4.5){\line(1,0){6}}
  \put(20,4.5){\line(1,0){6}}
  \put(56,4.5){\line(1,0){6}}
\end{picture}
}
\,
=\hat{1}_{\{j_1,\dots,2r\}}}}
    K_{\pi''}(L_{\sigma_{j_1'}}, L_{\sigma_{j_1'+1}},\dots,L_{\sigma_r},L_{\sigma_r})
\\
&=
\sum_{\sigma_{j_1'},\sigma_{j_1'+1},\dots,\sigma_r\in\SG_n} 
    K^{j_1'}_{B_1,B_2,\dots,B_{k-1}}(P,P,\dots,P,L_{\sigma_{j_1'}})
\\
&\phantom{=xxxxxxxxxx}\times
   \sum_{\substack{\pi''\in \NC^{B_k}_{B_1,B_2,\dots,B_{k-1}}(\{j_1,\dots,2r\}) \\
\\ \pi''\vee 
\,
{
\setlength{\unitlength}{0.7\unitlength}
\tiny
\begin{picture}(62,3.5)(1,0)
  \put(2,0){\line(0,1){4.5}}
  \put(8,0){\line(0,1){4.5}}
  \put(14,0){\line(0,1){4.5}}
  \put(20,0){\line(0,1){4.5}}
  \put(26,0){\line(0,1){4.5}}
\put(33.5,0){$\cdots$}
  \put(56,0){\line(0,1){4.5}}
  \put(62,0){\line(0,1){4.5}}
  \put(2,4.5){\line(1,0){0}}
  \put(8,4.5){\line(1,0){6}}
  \put(20,4.5){\line(1,0){6}}
  \put(56,4.5){\line(1,0){6}}
\end{picture}
}
\,
=\hat{1}_{\{j_1,\dots,2r\}}}}
    K_{\pi''}(L_{\sigma_{j_1'}}, L_{\sigma_{j_1'+1}},\dots,L_{\sigma_r},L_{\sigma_r})
\end{aligned}
.
\end{multline*}
And  by Lemma~\ref{lemm:sumyobcietelematgeneral} the factor
$$
K^{j_1'}_{B_1,B_2,\dots,B_{k-1}}(P,P,\dots,P,L_\sigma),
$$
vanishes for every $\sigma$. 
Here we use the notation  $j'_{i}=j_{\lceil i/2 \rceil}$ 
where $\lceil \cdot \rceil$ is the ceiling function 
which rounds up to the nearest integer.
%
\end{proof}

\begin{Rem}
  \begin{enumerate}
   \item 
    Note that in the case of sample variance we have to assume identical distribution
    of the involved random variables for the cancellation phenomenon to take place;
    in the case of the free commutator this requirement is not necessary.
   \item 
The argument put forward in the previous proof
is not valid in classical probability except in the case where
$B_k$ is an interval block.
For example if $r=3$ and $B=\{1,3,6\}$
then the block $B$ alone ensures that $\pi\vee  
\begin{picture}(32,3.5)(1,0)
  \put(2,0){\line(0,1){4.5}}
  \put(8,0){\line(0,1){4.5}}
  \put(14,0){\line(0,1){4.5}}
  \put(20,0){\line(0,1){4.5}}
  \put(26,0){\line(0,1){4.5}}
  \put(32,0){\line(0,1){4.5}}
  \put(2,4.5){\line(1,0){6}}
  \put(14,4.5){\line(1,0){6}}
  \put(26,4.5){\line(1,0){6}}
\end{picture}      
  =   
\begin{picture}(32,3.5)(1,0)
  \put(2,0){\line(0,1){4.5}}
  \put(8,0){\line(0,1){4.5}}
  \put(14,0){\line(0,1){4.5}}
  \put(20,0){\line(0,1){4.5}}
  \put(26,0){\line(0,1){4.5}}
  \put(32,0){\line(0,1){4.5}}
  \put(2,4.5){\line(1,0){30}}
\end{picture}$ 
and thus
$$
\sum_{
\substack{
  \pi\in \SP^{B}(6)\\
  \pi\vee 
\begin{picture}(32,3.5)(1,0)
  \put(2,0){\line(0,1){4.5}}
  \put(8,0){\line(0,1){4.5}}
  \put(14,0){\line(0,1){4.5}}
  \put(20,0){\line(0,1){4.5}}
  \put(26,0){\line(0,1){4.5}}
  \put(32,0){\line(0,1){4.5}}
  \put(2,4.5){\line(1,0){6}}
  \put(14,4.5){\line(1,0){6}}
  \put(26,4.5){\line(1,0){6}}
\end{picture}      
  =   
\begin{picture}(32,3.5)(1,0)
  \put(2,0){\line(0,1){4.5}}
  \put(8,0){\line(0,1){4.5}}
  \put(14,0){\line(0,1){4.5}}
  \put(20,0){\line(0,1){4.5}}
  \put(26,0){\line(0,1){4.5}}
  \put(32,0){\line(0,1){4.5}}
  \put(2,4.5){\line(1,0){30}}
\end{picture}
}}
K_\pi(T_1,T_1,T_2,T_2,T_3,T_3)
=K_3(T_1,T_2,T_3)\,\tau( T_1T_2T_3),
$$
where the sum runs over all set partitions.
\end{enumerate}

\end{Rem}

\begin{proof}[Proof of Theorem~\ref{twr:OddElement}]
Let  $\X_1,\dots,\X_n$ be free copies of a fixed random variable $X$.
We apply Lemma~\ref{lem:oddcancellation} for $L=X_1-\overline{X}$ and 
$P=\sum_{\sigma\in\SG_n} L_\sigma^2=(n-1)!\,Q_n$ 
to conclude from \eqref{eq:telescope} that
$$
K_r(Q_n) = 
\left(\frac{1}{(n-1)!}\right)^r
\sum_{\substack{\pi\in \NCeven(2r)\\ \pi\vee\onetwo{r}=\hat{1}_{2r}}}
\tilde{K}_\pi(L).
$$
Now the fact that $X_i$ are free identically distributed implies that in the sum
\eqref{eq:Ktilde}
every term either vanishes or is a multiple of $K_\pi(X)$.
Therefore we may write
 \begin{equation} 
\label{eq:Pom1DowodTwrGlo} 
K_r(\Q_n) =
    \sum_{\pi\in \NCeven_0(2r)}
    c_n(\pi)\,K_\pi(X)
\end{equation}
where $ \NCeven_0(2r) =\{\pi\in\NCeven(2r) \mid \pi \vee \onetwo{r}= \hat{1}_{2r}\}$ and $c_n(\pi)\in \R$. 
We will derive an explicit formula for these coefficients
in Section~\ref{sec:Rcyclic} below.
Using identity \eqref{eq:Pom1DowodTwrGlo}  we will show that all even cumulants of higher order vanish,
i.e., $K_{2i}(X)=0$ for $i\geq 2$.
First let us compute the parameters $c_n(\pi)$ in the extreme cases
$\pi=\pispecial{r}$
and 
$\pi=\hat{1}_{2r}$.

For $\pispecial{r}=
\begin{picture}(56,6.5)(1,0)
  \put(2,0){\line(0,1){7.5}}
  \put(8,0){\line(0,1){4.5}}
  \put(14,0){\line(0,1){4.5}}
  \put(20,0){\line(0,1){4.5}}
  \put(26,0){\line(0,1){4.5}}
  \put(44,0){\line(0,1){4.5}}
  \put(50,0){\line(0,1){4.5}}
  \put(56,0){\line(0,1){7.5}}
\put(28.0,0){$\cdots$}
  \put(8,4.5){\line(1,0){6}}
  \put(20,4.5){\line(1,0){6}}
  \put(44,4.5){\line(1,0){6}}
  \put(2,7.5){\line(1,0){54}}
\end{picture}$ only the
second cumulant  contributes to  $c_n(\pispecial{r})$ and
the value of the latter does not change if we replace
$X$ with a centered semicircular variable of variance $\sigma^2$.
In this case Corollary~\ref{corr:dwupartycje} implies that
$$
K_r(Q_n) = c_n(
\begin{picture}(56,6.5)(1,0)
  \put(2,0){\line(0,1){7.5}}
  \put(8,0){\line(0,1){4.5}}
  \put(14,0){\line(0,1){4.5}}
  \put(20,0){\line(0,1){4.5}}
  \put(26,0){\line(0,1){4.5}}
  \put(44,0){\line(0,1){4.5}}
  \put(50,0){\line(0,1){4.5}}
  \put(56,0){\line(0,1){7.5}}
\put(28.0,0){$\cdots$}
  \put(8,4.5){\line(1,0){6}}
  \put(20,4.5){\line(1,0){6}}
  \put(44,4.5){\line(1,0){6}}
  \put(2,7.5){\line(1,0){54}}
\end{picture})\,\sigma^{2r},
$$
and from Proposition~\ref{twr:Lehner:2003} we infer that
$c_n(
\begin{picture}(56,6.5)(1,0)
  \put(2,0){\line(0,1){7.5}}
  \put(8,0){\line(0,1){4.5}}
  \put(14,0){\line(0,1){4.5}}
  \put(20,0){\line(0,1){4.5}}
  \put(26,0){\line(0,1){4.5}}
  \put(44,0){\line(0,1){4.5}}
  \put(50,0){\line(0,1){4.5}}
  \put(56,0){\line(0,1){7.5}}
\put(28.0,0){$\cdots$}
  \put(8,4.5){\line(1,0){6}}
  \put(20,4.5){\line(1,0){6}}
  \put(44,4.5){\line(1,0){6}}
  \put(2,7.5){\line(1,0){54}}
\end{picture})=n-1$.


To compute the value of $c_n(\hat{1}_{2r})$ 
it is convenient to switch to tensor notation and to identify the multilinear
cumulant functional $K_r:\mathcal{A}^n\to \IC$ with
its linear extension $K_r:\mathcal{A}^{\otimes n}\to \IC$.
Let us now assume without loss of generality that $\sigma=1$.
We have to evaluate
\begin{align*}
\tilde{K}_{\hat{1}_{2r}}(L)
&= \sum_{i_1,i_2,\dots,i_r=1}^n K_{2r}(X_{i_1}-\overline{X},X_{i_1}-\overline{X},X_{i_2}-\overline{X},X_{i_2}-\overline{X},\dots,X_{i_r}-\overline{X},X_{i_r}-\overline{X})
\\
&= K_{2r}
\biggl(
\biggl(
  \sum_{i=1}^n (X_i-\overline{X})\otimes(X_i-\overline{X})
\biggr)^{\otimes r}
\biggr)
\\
&= K_{2r}
\biggl(
\biggl(
  \sum_{i=1}^n X_i\otimes X_i-n\overline{X}\otimes\overline{X}
\biggr)^{\otimes r}
\biggr).
\\
\end{align*}
Expanding this power yields cumulants of the form
\begin{align*}
K_{2r}\Bigl(
\Bigl(
\sum_{i=1}^n X_i\otimes X_i
\Bigr)^{\otimes k}
\otimes
\Bigl(
-n\overline{X}\otimes\overline{X}
\Bigr)^{\otimes(r-k)}
\Bigr)
&= n
  K_{2r}\Bigl(
\Bigl(
 X_1\otimes X_1
\Bigr)^{\otimes k}
\otimes
\Bigl(
-\frac{1}{n} X_1\otimes X_1
\Bigr)^{\otimes(r-k)}
\Bigr)
\\
&= 
n
\left(
-\frac{1}{n}
\right)^{r-k}
  K_{2r}(X),
\end{align*}
and in total
$$
\tilde{K}_{\hat{1}_{2r}}(L) 
= n\sum_{k=0}^r
    \binom{r}{k}
    \left(
      -\frac{1}{n}
    \right)^{r-k}
    K_{2r}(X)
= 
n
\left(
1-\frac{1}{n}
\right)^r 
K_{2r}(X) .
$$
Next, to evaluate even cumulants, equate  the $r$-th cumulants of $\Q_n$ and $\chi^2(n-1)$, i.e.,
$$K_r(\Q_n)=K_r(\chi^2(n-1))=n-1.$$ 
Denote  $\NCeven_0'(2r)=\NCeven_0(2r) \setminus \{
\begin{picture}(56,6.5)(1,0)
  \put(2,0){\line(0,1){7.5}}
  \put(8,0){\line(0,1){4.5}}
  \put(14,0){\line(0,1){4.5}}
  \put(20,0){\line(0,1){4.5}}
  \put(26,0){\line(0,1){4.5}}
  \put(44,0){\line(0,1){4.5}}
  \put(50,0){\line(0,1){4.5}}
  \put(56,0){\line(0,1){7.5}}
\put(28.0,0){$\cdots$}
  \put(8,4.5){\line(1,0){6}}
  \put(20,4.5){\line(1,0){6}}
  \put(44,4.5){\line(1,0){6}}
  \put(2,7.5){\line(1,0){54}}
\end{picture}
,
\hat{1}_{2r}
\}$,
then we have
 \begin{equation}
 \begin{aligned}
   n-1 &=\sum_{\pi\in \NCeven_0(2r)}c_n(\pi) \,K_\pi(X)
    \\ &= \sum_{\pi\in \NCeven_0'(2r)}c_n(\pi)\,K_\pi(X)
          +\sum_{ {\substack{\pi\in \NC_{2}(2r)\\ \pi\vee\onetwo{r}=\hat{1}_{2r}}}}c_n(\pi)
          \,K_\pi(X) + 
          c_n(\hat{1}_{2r})\,K_{2r}(X)
\\ &=
\sum_{\pi\in \NCeven_0'(2r)}c_n(\pi)\,K_\pi(X)
 + c_n(
\begin{picture}(56,6.5)(1,0)
  \put(2,0){\line(0,1){7.5}}
  \put(8,0){\line(0,1){4.5}}
  \put(14,0){\line(0,1){4.5}}
  \put(20,0){\line(0,1){4.5}}
  \put(26,0){\line(0,1){4.5}}
  \put(44,0){\line(0,1){4.5}}
  \put(50,0){\line(0,1){4.5}}
  \put(56,0){\line(0,1){7.5}}
\put(28.0,0){$\cdots$}
  \put(8,4.5){\line(1,0){6}}
  \put(20,4.5){\line(1,0){6}}
  \put(44,4.5){\line(1,0){6}}
  \put(2,7.5){\line(1,0){54}}
\end{picture})
 +\frac{(n-1)^{r}}{n^{r-1}}\,K_{2r}(X)
\\ &=
\sum_{\pi\in \NCeven_0'(2r)}c_n(\pi)\,K_\pi(X)
  +n-1
  +\frac{(n-1)^{r}}{n^{r-1}}\,K_{2r}(X).
 \end{aligned}  
 \end{equation}
This yields
\begin{equation}
\label{eq:PomDowod1} 
 \sum_{\pi\in \NCeven_0'(2r)}c_n(\pi)\,K_\pi(X)+\frac{(n-1)^{r}}{n^{r-1}}\,K_{2r}(X)=0.
\end{equation}  
%
and the blocks of any $\pi\in\NCeven_0'(2r)$ have size strictly smaller
than $2r$, it follows by induction that
 $K_{2r}(X)=0$ for $r\geq 2$. 

Conversely, 
suppose that $\X_i$'s are odd, then from Lemma~\ref{lem:oddcancellation}
 we get 
\begin{align*}
  K_r(\Q_n) 
  &=\sum_{\pi\in \NCeven_0(2r)}
     c_n(\pi)\,K_\pi(X)
   =\sum_{\substack{\pi\in \NC_{2}(2r)\\ \pi\vee\onetwo{r}=\hat{1}_{2r}}}c_n(\pi)\,K_\pi(X)
 \\
 &= c_{n}(
\begin{picture}(56,6.5)(1,0)
  \put(2,0){\line(0,1){7.5}}
  \put(8,0){\line(0,1){4.5}}
  \put(14,0){\line(0,1){4.5}}
  \put(20,0){\line(0,1){4.5}}
  \put(26,0){\line(0,1){4.5}}
  \put(44,0){\line(0,1){4.5}}
  \put(50,0){\line(0,1){4.5}}
  \put(56,0){\line(0,1){7.5}}
\put(28.0,0){$\cdots$}
  \put(8,4.5){\line(1,0){6}}
  \put(20,4.5){\line(1,0){6}}
  \put(44,4.5){\line(1,0){6}}
  \put(2,7.5){\line(1,0){54}}
\end{picture}
) =n-1. 
\end{align*}

\end{proof}

\noindent  \textit{Proof of Proposition \ref{twr:kruglow}}. 
Recall that as a consequence of the free L\'evy-Khinchin formula
 (see for example \cite[Theorem~13.16]{NicaSpeicher:2006}) 
the random variable $\X_i$ is freely infinitely divisible if and only if 
$$K_{n+2}(\X_i)=\int_{\IR}x^n d\rho_i(x),$$ 
for some positive finite measure $\rho_i(x)$ on $\R$. 
For the semicircular distribution the measure is $\rho_i=\delta_0$ and it
suffices to show that $\int x^2d\rho_i(x)=K_4(X_i)=0$.
Now if $\tau(Y_1)=\tau(Y_2)=\tau(Y_3)=\tau(Y_4)=0$ 
then the product formula from Theorem~\ref{thm:krawczyk} implies
\begin{equation}
  \label{eq:krawczykQ}
  K_2(Y_1Y_2,Y_3Y_4)=K_2(Y_1,Y_4)\,K_2(Y_2,Y_3)+K_4(Y_1,Y_2,Y_3,Y_4)
  .
\end{equation}
We will apply this to $Q_n$, so let us first compute the cumulants 
which will appear after evaluation of \eqref{eq:krawczykQ}.
By assumption $K_2(X_i)=1$ for all $i$ and therefore the covariances are
$$
K_2(\X_{i}-\overline{\X},\X_{j}-\overline{\X})=
\begin{cases}
\frac{n-1}{n} &\text{if $i=j$,}  \\
-\frac{1}{n} &\text{if $i\neq j$.} 
\end{cases}
$$    
It remains to consider cumulants of order 4.
First,
\begin{align*}
  \sum_{i=1}^n K_4(X_i-\overline{X})
  &= \sum_{i=1}^n 
      \biggl(
        \left(
          1-\frac{1}{n}
        \right)^4
        K_4(X_i)
        +
        \sum_{l\ne i}
        \left(
          -\frac{1}{n}
        \right)^4
        K_4(X_l)
      \biggr)
      \\
  &=
  \left(
    \left(
      1-\frac{1}{n}
    \right)^4
    +
    \frac{n-1}{n^4} 
  \right)
  \sum_{i=1}^n   K_4(X_i),
\end{align*}
second,
\begin{multline*}
  \sum_{\substack{i,j=1\\ i\ne j}}^n 
    K_4(X_i-\overline{X},X_i-\overline{X},X_j-\overline{X},X_j-\overline{X})
    \\
    \begin{aligned}
      &= 
      \sum_{\substack{i,j=1\\ i\ne j}}^n 
      \biggl(
      \left(
        1-\frac{1}{n}
      \right)^2
      \left(
        -\frac{1}{n}
      \right)^2
      (
      K_4(X_i)
      +
      K_4(X_j)
      )
      +
      \sum_{l\ne i,j}
      \left(
        -\frac{1}{n}
      \right)^4
      K_4(X_l)
      \biggr)
      \\
      &= 
      \left(
        2(n-1)
        \left(
          1-\frac{1}{n}
        \right)^2
        \left(
          \frac{1}{n}
        \right)^2
        +
        \frac{(n-1)(n-2)}{n^4}
      \right)
      \sum_{i=1}^n
       K_4(X_i),
    \end{aligned}
\end{multline*}
and thus
$$
  \sum_{\substack{i,j=1}}^n 
    K_4(X_i-\overline{X},X_i-\overline{X},X_j-\overline{X},X_j-\overline{X})
    =
    \frac{(n-1)^2}{n^2}
    \sum_{i=1}^n
     K_4(X_i).
$$
Using these formulas
we now proceed to  \eqref{eq:krawczykQ}
and obtain
\begin{align*} 
  n-1
  &= K_2(\Q_n,\Q_n)
  \\
  &=\sum_{i,j=1}^n K_2((\X_i-\overline{\X})^2,(\X_j-\overline{\X})^2)
  \\
  &=\sum_{i,j=1}^n K_2(\X_i-\overline{\X},\X_j-\overline{\X})\,
                  K_2(\X_j-\overline{\X},\X_i-\overline{\X})
\\ &\phantom{===}
    +
    \sum_{i,j=1}^n K_4(\X_i-\overline{\X},\X_i-\overline{\X},
                     \X_j-\overline{\X},\X_j-\overline{\X})
  \\
  &=\sum_{i=1}^n [K_2((\X_i-\overline{\X}),(\X_i-\overline{\X}))]^2+\sum_{i=1,j=1, i\neq j}^n[K_2(\X_i-\overline{\X},\X_j-\overline{\X})]^2
\\
 &\phantom=+\sum_{i=1}^n K_4(\X_i-\overline{\X})+\sum_{i=1,j=1, i\neq j}^nK_4(\X_i-\overline{\X},\X_i-\overline{\X},\X_j-\overline{\X},\X_j-\overline{\X}) \\
 &= \frac{(n-1)^2}{n}+\frac{n-1}{n}+\frac{(n-1)^2}{n^2}\sum_{i=1}^nK_4(\X_i). 
\end{align*}
So we see that $\sum_{i=1}^nK_4(\X_i)=\sum_{i=1}^n\int_{\R}x^2 d\rho_i(x)=0$ and thus $\rho_i(x)=\delta_0(x)$. 
Note that the above measure is the free L\'evy measure  of the semicircle distribution with mean zero, and variance
one. \begin{flushright} $\square$ \end{flushright}

\begin{proof}[Proof of Proposition \ref{twr:OdpowidenikKaganLetac}]
Part (1).
We write $a=\norm{a}\theta=\norm{a}(\theta_1,\dots,\theta_n)$ where  $\theta$  belongs to the unit sphere
of $E$, i.e., $\sum\theta_i=0$. 
Thus for $\norm{a}>0,$ $a_i=\norm{a}\theta_i$ and $r\geq 2$ we have
\begin{equation} 
\label{eq:kumualntydonieskonczonosci}
  \begin{aligned}
    \frac{K_r(\sum_{i=1}^n(\X_i-\overline{\X}+a_i)^2)}{\norm{a}^r}
  &=K_r\Big(\sum_{i=1}^n(\X_i-\overline{\X})^2/\norm{a}
            -\sum_{i=1}^n2(\X_i-\overline{\X})\,\theta_i+\norm{a}\Big)\\
  &=K_r\Big(\sum_{i=1}^n(\X_i-\overline{\X})^2/\norm{a}-\sum_{i=1}^n2(\X_i-\overline{\X})\,\theta_i\Big)
\\
&=K_r\Big(\sum_{i=1}^n(\X_i-\overline{\X})^2/\norm{a}-\sum_{i=1}^n2\X_i\theta_i\Big).
  \end{aligned}
\end{equation}
 By the hypothesis the left hand side of \eqref{eq:kumualntydonieskonczonosci}  does not depend on $\theta$, and thus the limit on the right hand side
\begin{align*} 
\lim_{\norm{a}\to + \infty}\frac{K_r(\sum_{i=1}^n(\X_i-\overline{\X}+a_i)^2)}{\norm{a}^r}&=K_r\Big(-\sum_{i=1}^n2\X_i\theta_i\Big),
\end{align*}
does not depend on $\theta$ either.
Now freeness implies that
 \begin{align} 
S_r(\theta_1,\dots,\theta_n):=K_r\Big(\sum_{i=1}^n\X_i\theta_i\Big)=\sum_{i=1}^n\theta_i^rK_r(\X_i)=\Big(\sum_{i=1}^n\theta_i^r\Big)K_r,
\end{align}
is a constant function on the unit sphere 
of the space $E$. Thus we see $S_2(\theta_1,\dots,\theta_n)=K_2$ and $S_r$ for $r\geq 3$ is constant function on the unit sphere 
of the space $E$ if and only $K_r=0$  for $r\geq 3$.

We now show part $(2)$ of Proposition \ref{twr:OdpowidenikKaganLetac}.  It's easy to observe that for $\theta\in E$ we have 
 \begin{align} 
K_r\Big(\sum_{i=1}^n(\X_i-\overline{\X})\,\theta_i\Big)=\sum_{i=1}^n\theta_i^rK_r(\X_i)=\Big(\sum_{i=1}^n\theta_i^r\Big)K_r,
\end{align}
is constant function on the unit sphere 
of the space $E$ and from the above consideration we get the statement.

\end{proof}
\begin{Rem}
  The preceding proof is also valid in the general setting of
  \cite{Lehner:2003}.
\end{Rem}

\section{$R$-cyclic matrices and free infinite divisibility of quadratic forms}
\label{sec:Rcyclic}
In this section we  show how the sample variance can be combined with
the concept of $R$-cyclicity and give a recipe for computing the coefficients
$c_n(\pi)$ in equation \eqref{eq:Pom1DowodTwrGlo}. In particular we also show
that sample variance  preserves free infinite divisibility.

\subsection{$R$-cyclic matrices and the distribution of sample variance}


The concept of $R$-cyclicity was introduced by Nica, Shlyakhtenko and Speicher
\cite{NicaSpeicherShlyakhtenko:2002}. Our aim is now to exhibit its relation
to the sample variance and other quadratic forms.  
We show that the theory of  $R$-cyclic matrices 
can be used to compute the distribution
of the sample variance and  give  a formula for the cumulants of the sample
variance  in terms of the even cumulants, which generalizes and unifies
two types of results, namely cumulants of squares of even elements
\cite[Proposition~11.25]{NicaSpeicher:2006} and cumulants of quadratic
forms in gaussian random variables \cite[Proposition~4.4]{Lehner:2003}. 

Here we consider matrices over a non-commutative probability space.
Let $(\A,\tau)$ be a non-commutative probability
space, and let $n$ be a positive integer. 
The algebra $M_n(\A)$
of $n\times n$ matrices over $\A$ is a noncommutative probability space
with canonical expectation functional
$$\tau_n(A)=\frac{1}{n}\sum_{i=1}^n\tau(a_{i,i}),$$
where $A=[a_{i,j}]_{i,j=1  }^n$ is  a matrix in $M_n(\A)$.  Then $(M_n(\A),\tau_n)$ is itself a non-commutative probability space.
The definition of $R$-cyclicity is in terms of
the joint $R$-transform of the entries of the matrix: one requires that
only the cyclic non-crossing cumulants of the entries are allowed to be
different from $0$, see Definition~\ref{def:Rcyclic} below.
Equivalently, it was shown in \cite[Theorem~8.2]{NicaSpeicherShlyakhtenko:2002}
that $R$-cyclicity is equivalent to the fact that $A$ is free from $M_n(\IC)$
with amalgamation over the algebra $\mathcal{D}_n$ of scalar diagonal matrices
with respect to the conditional expectation
\begin{equation}
  \label{eq:ED}
\begin{aligned}
  \ED:M_n(\mathcal{A})&\to M_n(\IC)\\
  A &\mapsto \sum_{i=1}^n E_i\tau^{(n)}(A)E_i,
\end{aligned}
\end{equation}
where by $E_i$ we denote the projection matrix onto the $i$-th unit vector
and $\tau^{(n)}(A)_{ij}=\tau(a_{i,j})$ is the entry-wise trace.

\begin{defi}\label{def:Rcyclic}
  Let $(M_n(\A),\tau_n)$ and $n$ be as above, 
  then a matrix $\bX=[\X_{i,j}]_{i,j=1  }^n\in M_n(\A)$. 
  is said to be   \emph{$R$-cyclic} if 
  for every $r\geq 1$ and for every choice of indices
  $1\leq i_1,j_1,\dots,i_r,j_r\leq n $ the cumulant
$$
K_r(\X_{i_1,j_1},\X_{i_2,j_2},\dots,\X_{i_r,j_r})=0,$$
unless the indices are cyclic in the sense
that $ j_1=i_2,j_2=i_3,\dots,j_{r-1}=i_r,j_r=i_1$. 
In this case the formal noncommutative power series
\begin{equation}
  \label{eq:determiningseries}
f_{\bX}(z_1,z_2,\dots,z_n)
=\sum_{r=1}^\infty\sum_{i_1,i_2,\dots,i_r=1}^n
   K_r(\X_{i_r,i_1},\X_{i_1,i_2},\dots,\X_{i_{r-1},i_r})
   \,
   z_{i_1}z_{i_2}\cdots z_{i_r},
\end{equation}
is called the \emph{determining series} of the entries of $\bX$.
\end{defi}
The concept of $R$-cyclicity generalizes the concept of $R$-diagonality
\cite[Ex.~20.5]{NicaSpeicher:2006} in the sense that $X$ is $R$-diagonal
if and only if the matrix 
$
\left[
  \begin{smallmatrix}
    0&X\\
    X^*&0
  \end{smallmatrix}
\right]
$
is $R$-cyclic.

\begin{lemm}
  \label{lem:kreweras}
  For scalar matrices $A\in M_n(\IC)$ we have
  \begin{enumerate}[(i)]
   \item \label{it:kreweras1}
    $$
      \sum_{i=1}^n E_iA_1E_iA_2\dotsm E_iA_rE_i
    = \ED(A_1)\ED(A_2)\dotsm \ED(A_r).
    $$
   \item  \label{it:kreweras2}
    Let $\pi\in\NC(r)$, then
    $$
    \sum_{\ker\underline{i}\geq\pi} \ED (A_1E_{i_1}A_2E_{i_2}\dotsm A_rE_{i_r}A_{r+1})
    =\ED[\Krewlr\pi] (A_1,A_2,\dots,A_{r+1}).
    $$
  \end{enumerate}
\end{lemm}
\begin{proof}
  Part~\eqref{it:kreweras1} follows immediately from the expansion
  $$
  \sum_{k=1}^n E_kA_1E_kA_2\dotsm A_rE_k
  = \sum_{k=1}^n E_{k} a^{(1)}_{k,k}a^{(2)}_{k,k}\dotsm a^{(r)}_{k,k}.
  $$
  To see part~\eqref{it:kreweras2} we single out the last block
  of $\pi$ (i.e., the block containing $r$, see Lemma~\ref{lemm:kreweras}),
  say
  $B=\{j_1<j_2<\dotsm <j_p=r\}$,
  and group the remaining blocks into
  subpartitions, empty partitions allowed,
  say $\pi_1\in\NC([1,j_1-1])$, 
  $\pi_2\in\NC([j_1+1,j_2-1]),\dots,\pi_p\in\NC([j_{p-1}+1,j_p-1])$.
  Then we have
  $$
    \ED\biggl(
       \sum_{\ker\underline{i}\geq\pi}  (A_1E_{i_1}A_2E_{i_2}\dotsm E_{i_r} A_{r+1})
       \biggr)
       =
    \ED\biggl(
         \sum_i
         A_1'E_iA_2'E_i\dots A_p'E_i A_{r+1}
       \biggr),
  $$
  where
  $$
  A_k'=  \sum_{\ker\underline{i}\geq \pi_k} 
  A_{j_{k-1}+1}E_{i_1}
  A_{j_{k-1}+2}E_{i_2}
  \dotsm
  A_{j_{k}}.
  $$
  By part~\eqref{it:kreweras1} this is
  $$
  \ED(A_1'\ED(A_2') \dots\ED(A_p') A_{r+1}),
  $$
  and by induction this is 
  $$
  \ED(A_1' \ED[\Krewlr{\pi_2}](A_{j_1+1},\dots,A_{j_2}) 
           \dotsm 
           \ED[\Krewlr{\pi_p}](A_{j_{p-1}+1},\dots,A_r) A_{r+1})
  = \ED[\Krewlr\pi](A_1,A_2,\dots,A_{r+1}),
  $$
  where we used Lemma~\ref{lemm:kreweras}.
\end{proof}
\begin{prop}
  \label{prop:Rcyclic}
  Let $\X_1, \X_2,\dots, \X_n\in \A$ be a free family of even
  random variables and $A=[a_{i,j}]_{i,j=1}^n\in M_n(\C)$ a scalar matrix.
  Then
    the Hadamard product matrix
    \begin{equation}
      \label{eq:ZAX}
      \mathbf{Z}=A\circ\bX= [a_{i,j}X_iX_j]_{i,j=1  }^n= 
      \begin{bmatrix}
        a_{1,1}\X_1^2 & a_{1,2}\X_1\X_2 & \dots  &  a_{1,n}\X_1\X_n  \\ 
        a_{2,1}\X_2\X_1 &  a_{2,2}\X_2^2 & \dots & a_{2,n}\X_2\X_n 
        \\
        \multicolumn{4}{c}{\dotfill}\\
        a_{n,1}\X_n\X_1 &  a_{n,2}\X_n\X_2 & \dots &   a_{n,n}\X_n^2
      \end{bmatrix},
    \end{equation}
    is  $R$-cyclic.
\end{prop}
\begin{proof}
    We make use of the product formula of Theorem~\ref{thm:krawczyk}
    and Lemma~\ref{lemm:lematoparzystych} to compute
    \begin{equation}
      \label{eq:KrHx1Xh2}
    \begin{aligned}
      K_r(X_{i_1}X_{i_2},
         X_{i_3}X_{i_4},
         \dots,
         X_{i_{2r-1}}X_{i_{2r}})
&= \sum_{\substack{\pi\in\NCeven(2r)\\
    \pi\vee \onetwo{r} = \hat{1}_{2r}}}
      K_\pi(X_{i_{1}},X_{i_{2}},\dots,X_{i_{2r}})
\\
&= \sum_{\substack{\pi\in\NCeven(2r)\\
    \pi\geq \pispecial{r}}}
      K_\pi(X_{i_{1}},X_{i_{2}},\dots,X_{i_{2r}})
    \end{aligned}
    \end{equation}
    and by \eqref{eq:kerh>=pi}
    these mixed cumulants vanish unless $\ker \underline{i}\geq\pispecial{r}$,
    i.e., unless $i_{1}=i_{2r}$ and $i_{2j}=i_{2j+1}$ for all $j$, 
    which exactly means $R$-cyclicity. It is easy to see that the same holds
    for $Z_{i,j}=a_{i,j}X_iX_j$.
\end{proof}

\begin{Rem}
  In some sense Proposition~\ref{prop:Rcyclic}
  is a generalization of the fact 
  \cite[Theorem~20.6]{NicaSpeicher:2006} 
  that the product of two free even selfadjoint elements is $R$-diagonal.
  This fact is indeed a consequence if we put 
  $A=
  \left[
    \begin{smallmatrix}
    0&1\\
    1&0
    \end{smallmatrix}
  \right]
  $
  in the preceding proposition.
  In fact it was shown in \cite{HaagerupLarsen:1998} that
  every $R$-diagonal element can be written as a product
  of two free even selfadjoint elements.
  It is an interesting question what would be a natural factorization
  of $R$-cyclic matrices. While it is necessary for a
  matrix to be $R$-cyclic that its entries form  $R$-diagonal pairs, 
  example \cite[Ex.~20.6]{NicaSpeicher:2006} shows that
  the representation \eqref{eq:ZAX} in the preceding
  proposition does not cover all $R$-cyclic matrices.
\end{Rem}

\begin{prop} \label{prop:CykliczneVariancja}
  Let $\X_1, \X_2,\dots, \X_n\in \A$ be a free family of even random variables, $\bX= [X_iX_j]_{i,j=1  }^n$ and $A=[a_{i,j}]_{i,j=1}^n\in M_n(\C)$ a scalar matrix.
  \begin{enumerate}[(i)]
   \item \label{it:cyclic2}
    The determining series of the entries of the $R$-cyclic matrix
    $\bZ=A\circ\bX$ and the $R$-transform of the
    quadratic form $T_n=\sum_{i,j}^na_{i,j}\X_i\X_j$ are related by
    \begin{equation}
      \label{eq:cyclic2}
      f_{A\circ\bX}(z,\dots,z)=\mathcal{R}_{T_n}(z),
    \end{equation}
    where $\mathcal{R}_{T_n}(z)=zR_{T_n}(z)$.
   \item \label{it:cyclic3}
    The cumulants of $T_n$ are given by
\begin{align}  \label{eq:kumulantsamplevariancenotiid}
 \nonumber &K_r(T_n)\\&=\sum_{i_1,\dots,i_r\in[n]} 
            \Tr(AE_{i_1}AE_{i_2}\dots AE_{i_r})\,
            \sum_{\substack{ \pi\in \NCeven(2r)\\ 
                \pi \vee \onetwo{r}=\hat{1}_{2r}}}
            K_\pi(X_{i_r},X_{i_1},X_{i_1},X_{i_2},\dots,X_{i_{r-1}},X_{i_r}).
    \end{align}
    \item If we assume in addition that $X_i$ are identically distributed
     the previous formula simplifies to the following convolution-like expression
    \begin{equation}  \label{eq:kumulantsamplevariance}
      K_r(T_n)=\sum_{ \pi\in \NC(r)}
      \Tr(\ED[\Krewl{\pi}](A)) \prod_{B\in\pi}K_{2|B|}(X).
    \end{equation} 
  \end{enumerate}
\end{prop}
\begin{proof}

From the definition of $T_n$ we see that 
\begin{align*}
 K_r(T_n)
 &= \sum_{i_1,i_2,\dots,i_{2r}\in[n]}
      K_r(Z_{i_{1},i_{2}},
         Z_{i_{3},i_{4}},
         \dots,
         Z_{i_{2r-1},i_{2r}})
  \\
 &= \sum_{\substack{
     i_1,i_2,\dots,i_{2r} \in[n]    \\
     \ker \underline{i}\geq\pispecial{r} }}
    \sum_{\substack{\pi\in\NCeven(2r)\\ \pi\geq \pispecial{r}}}
      a_{i_{1},i_{2}}      a_{i_{3},i_{4}} \dotsm       a_{i_{2r-1},i_{2r}}
      K_\pi(X_{i_{1}},X_{i_{2}},\dots,X_{i_{2r}}),
\intertext{where we used \eqref{eq:KrHx1Xh2}. 
  Having eliminated the zero contributions 
  we can apply Lemma~\ref{lemm:lematoparzystych} in the reverse direction and obtain}
 &= \sum_{\substack{
     i_1,i_2,\dots,i_{2r}  \in[n]   \\
     \ker \underline{i}\geq\pispecial{r} }}
      K_r(Z_{i_{1},i_{2}},
         Z_{i_{3},i_{4}},
         \dots,
         Z_{i_{2r-1},i_{2r}})
\\
 &= \sum_{i_1,i_2,\dots,i_r}
      K_r(Z_{i_r,i_1},Z_{i_1,i_2},\dots,Z_{i_{r-1},i_r}),
\intertext{
  which after comparison with \eqref{eq:determiningseries} yields
  \eqref{eq:cyclic2}. 
  We now expand further and obtain}
 &= \sum_{i_1,i_2,\dots,i_r\in[n]}
      a_{i_r,i_1} a_{i_1,i_2} \dotsm a_{i_{r-1,i_r}}
      K_r(X_{i_r}X_{i_1},X_{i_1}X_{i_2},\dots,X_{i_{r-1}}X_{i_r})
\\
 &=\sum_{i_1,\dots,i_r\in[n]} 
            \Tr(AE_{i_1}AE_{i_2}\dots AE_{i_r})\,
            \sum_{\substack{ \pi\in \NCeven(2r)\\ 
                \pi \vee \onetwo{r}=\hat{1}_{2r}}}
            K_\pi(X_{i_r},X_{i_1},X_{i_1},X_{i_2},\dots,X_{i_{r-1}},X_{i_r}),
 \\
\intertext{which yields \eqref{eq:kumulantsamplevariancenotiid}. Now denoting by $\hat\pi$ the image of $\pi\in\NC(r)$ under the bijection introduced in Lemma~\ref{lemm:lematoparzystych} we can rewrite this as}
 &=\sum_{ \pi\in \NC(r)}
 \left(\sum_{\ker\underline{i}\geq \pi}\Tr(AE_{i_1}AE_{i_2}\dots AE_{i_r})\right)
 K_{\hat\pi}(X)
.
\end{align*} 
Finally we infer \eqref{eq:kumulantsamplevariance}  from Lemma~\ref{lem:kreweras}.

\end{proof}

\begin{Rem}
  It was observed in \cite[Rem.~4.1]{NicaSpeicherShlyakhtenko:2002} that
  $R$-cyclicity is preserved under Hadamard products with constant matrices.
  Moreover inspecting the preceding proof one can easily see the that
  for an arbitrary $R$-cyclic matrix $\bX=[X_{i,j}]$ and any scalar matrix
  $A=[a_{i,j}]$
  the determining series of the Hadamard product $A\circ \bX=[a_{i,j}X_{i,j}]$
  is given by
  $$
  f_{A\circ \bX}(z_1,z_2,\dots,z_n)
  = \Tr(f_\bX(AE_1\otimes z_1,
            AE_2\otimes z_2,
            \dots,
            AE_n\otimes z_n))
  .
  $$
\end{Rem}

In fact we have proved the following slightly more general statement.
\begin{theo}
  Let $X_i$ be free copies of an even random variable $X$, 
  $\bX=[X_iX_j]_{i,j=1}^n$ be the matrix of products as above (which is $R$-cyclic)
  and let $A_1,A_2,\dots,A_r \in M_n(\C)$ be
  arbitrary scalar matrices. Then $(A_1\circ \bX,A_2\circ \bX,\dots,A_r\circ \bX)$
  is an $R$-cyclic family and the joint cumulant of $T_k=\sum_{ij}a^{(k)}_{ij}X_iX_j$ is
  $$
  K_r(T_1,T_2,\dots,T_r)=\sum_{ \pi\in \NC(r)}
      \Tr(\ED[\Krewl{\pi}](A_1,A_2,\dots,A_r))
      \prod_{B\in\pi}K_{2|B|}(X).
  $$
\end{theo}

It was shown in \cite[Section~8]{NicaSpeicherShlyakhtenko:2002} that
$R$-cyclicity of a matrix is equivalent to freeness
from the algebra of constant matrices $M_n(\IC)$ with amalgamation
over the commutative subalgebra $\mathcal{D}_n$ of constant diagonal matrices.
Moreover, the cyclic scalar cumulants can be interpreted as entries 
of the $\mathcal{D}_n$-valued cumulants as follows.
\begin{prop}[{\cite[Theorem~7.2]{NicaSpeicherShlyakhtenko:2002}}]
  \label{prop:NSS72}
  Let $(\bX_i)\subseteq M_n(\mathcal{A})$ be an $R$-cyclic family over 
  some noncommutative probability space $(\mathcal{A},\tau)$ and
  denote by $K_r^{\mathcal{D}}$ the operator valued cumulant functionals
  with respect to the conditional expecation \eqref{eq:ED}.
  Then for any $\Lambda_1,\Lambda_2,\dots,\Lambda_{r-1}\in\mathcal{D}_n$
  we have
  \begin{multline*}
  K_r^{\mathcal{D}}(\bX_1\Lambda_1,\bX_2\Lambda_2,\dots,\bX_{r-1}\Lambda_{r-1},\bX_r)
\\
  = \sum_{i_1,i_2,\dots,i_r=1}^n 
  \lambda_{i_1}^{(1)}  \lambda_{i_2}^{(2)}\dotsm   \lambda_{i_{r-1}}^{(r-1)}
  K_r(X_{i_r,i_1}^{(1)},X_{i_1,i_2}^{(2)},\dots,X_{i_{r-2},i_{r-1}}^{(r-1)},X_{i_{r-1},i_r}^{(r)})
  E_{i_r}.
  \end{multline*}
\end{prop}
In our context this leads to an operator valued boxed convolution
in the sense of \cite[Definition~2.1.6]{Speicher:1998} as follows.
\begin{prop}
  Let $\X_1, \X_2,\dots, \X_n\in \A$ be free even  copies
of a random variable $X$ and
  let $\bX=[X_iX_j]_{i,j=1}^n$.
  Then for any scalar matrices $A_1,A_2,\dots,A_r\in M_n(\IC)$ and
  $\Lambda_1,\Lambda_2,\dots,\Lambda_{r-1}\in\mathcal{D}_n$ we have
  \begin{multline*}
  K_r^{\mathcal{D}}(A_1\circ\bX\Lambda_1,A_2\circ\bX\Lambda_2,\dots,A_{r-1}\circ\bX\Lambda_{r-1},A_r\circ\bX)
  \\
  = \sum_{\pi\in \NC(r)}
\ED[\Krewl{\pi}](A_1\Lambda_1,A_2\Lambda_2,\dots,A_{r-1}\Lambda_{r-1},A_r)
      \prod_{B\in\pi}K_{2|B|}(X).
  \end{multline*}
\end{prop}
\begin{proof}
  We use Proposition~\ref{prop:NSS72} and expand
  \begin{align*}
\MoveEqLeft{
    K_r^{\mathcal{D}}(A_1\circ\bX\Lambda_1,A_2\circ\bX\Lambda_2,\dots,A_{r-1}\circ\bX\Lambda_{r-1},A_r\circ\bX)    }&
\\
&= \sum_{i_1,i_2,\dots,i_r=1}^n
    a_{i_r,i_1}^{(1)}
    \lambda_{i_1}^{(1)}
    a_{i_1,i_2}^{(2)}
    \lambda_{i_2}^{(2)}
    \dotsm
    a_{i_{r-1},i_r}^{(r)}
    K_r(X_{i_r}X_{i_1},X_{i_1}X_{i_2},\dots,X_{i_{r-1}}X_{i_r}) E_{i_r}
\\
&= \sum_{i_1,i_2,\dots,i_r=1}^n
    a_{i_r,i_1}^{(1)}
    \lambda_{i_1}^{(1)}
    a_{i_1,i_2}^{(2)}
    \lambda_{i_2}^{(2)}
    \dotsm
    a_{i_{r-1},i_r}^{(r)}
    \sum_{\substack{\pi\in\NC(2r)\\\pi\geq  \pispecial{r}}}
    K_\pi(X_{i_r},X_{i_1},X_{i_1},X_{i_2},\dots,X_{i_{r-1}},X_{i_r}) E_{i_r}
\\
&=     \sum_{\substack{\pi\in\NC(2r)\\\pi\geq  \pispecial{r}}}
    \sum_{\ker\underline{i}\geq\pi}
    a_{i_r,i_1}^{(1)}
    \lambda_{i_1}^{(1)}
    a_{i_1,i_2}^{(2)}
    \lambda_{i_2}^{(2)}
    \dotsm
    a_{i_{r-1},i_r}^{(r)}
    E_{i_r}
    K_{\pi}(X)
\\
&= \sum_{\pi\in\NC(r)}
\ED[\Krewl{\pi}](A_1\Lambda_1,A_2\Lambda_2,\dots,A_{r-1}\Lambda_{r-1},A_r)
 K_{\hat\pi}(X),
  \end{align*}
  where $\hat\pi$ is defined in the proof of Proposition \ref{prop:CykliczneVariancja}.
\end{proof}
\begin{Rem}
In fact it is easy to see that the matrix
$\Xi=\diag(X_1,X_2,\dots,X_n)$ is free from $M_n(\IC)$ with
amalgamation over $\mathcal{D}_n$ as well
\cite[Example~2.3]{NicaSpeicherShlyakhtenko:2002}.
We have shown above that $A\circ\bX=\Xi A\Xi$ has the same property
although $A$ has not.
\end{Rem}

As a final corollary we obtain the following formula
for the cumulants of the sample variance.
\begin{cor}  
\label{cor:KrQn}
Let $\X_1, \X_2,\dots, \X_n$ be free copies
of a random variable $X$ and $Q_n=nS_n^2$ the rescaled sample variance
defined in \eqref{eq:SampleVariance}.
Let  
$\tilde{X}$ be the symmetrization of $X$, i.e., a formal random variable with
even distribution and cumulants $K_{2r}(\tilde X)=K_{2r}(X)$.
Then 
\begin{equation}
 K_r(\Q_n)= (n-1)K_r(Z^2),
 \end{equation}
where
$Z=\sqrt{\frac{n}{n-1}}P\tilde{X}P$ is the free compression of
 the symmetrization 
$\tilde{X}$ of $X$ by a projection $P$ of trace $\tau(P)=\frac{n-1}{n}$.

\end{cor}
\begin{proof}
  By Lemma~\ref{lem:oddcancellation} the distribution of $Q_n$ does not
  change if we drop the odd cumulants and replace  $X$ 
  by its symmetrization $\tilde{X}$.
  The symmetrization $\tilde{X}$ being even, it follows from 
  Proposition~\ref{prop:CykliczneVariancja} that
  the information about the distribution of the sample variance is contained in
  the $R$-cyclic matrix
  $A\circ \tilde{\bX}=[a_{ij}\tilde{X}_i\tilde{X}_j]_{i,j=1}^n$,
  where  $A=I-\frac{1}{n}\One$ and $\One$ is the $n\times n$ matrix all of
  whose entries are $1$.
  This matrix is idempotent with $\ED(A)=(1-1/n)I$ and therefore
  $\ED[\pi](A)=(1-\frac{1}{n})^{\abs\pi}I$
  for every $\pi\in\NC(r)$. We insert this 
  into~\eqref{eq:kumulantsamplevariance} and 
  the cumulants of $Q_n$ evaluate to
  \begin{align*} %
    K_r(Q_n)
    &= n
       \sum_{\pi\in\NC(r)} 
       \left(
         1-\frac{1}{n}
       \right)^{\abs{\Krewl\pi}}
       \prod_{B\in\pi}K_{2\abs{B}}(\tilde{X})
    \intertext{This in turn by \eqref{eq:cardKrew} is equal to} 
    &= n
       \left(
         1-\frac{1}{n}
       \right)^{r+1} 
       \sum_{\pi\in\NC(r)} 
       \prod_{B\in\pi}\frac{n}{n-1}K_{2\abs{B}}(\tilde{X}).
  \end{align*}
  
  \begin{align*} %
   &= (n-1) 
       \sum_{\pi\in\NC(r)} 
       \prod_{B\in\pi}\frac{n}{n-1}
       K_{2\abs{B}}\Bigg( \sqrt{
         1-\frac{1}{n}
       }\tilde{X}\Bigg).
  \end{align*}
  In view of~\eqref{eq:KX2} this is the same as the cumulant $K_r(Z^2)$ where
  $Z$ is an even random variable with cumulants
  $$
  K_r(Z)=\frac{n}{n-1}K_r\Bigg( \sqrt{
         1-\frac{1}{n}
       }\tilde{X}\Bigg)
  .
  $$
  Such a random variable can be modeled as a free compression
  $$
  Z=\frac{n}{n-1}P \sqrt{
         1-\frac{1}{n}
       }\tilde{X}P=\sqrt{
        \frac{n}{n-1}
       }P \tilde{X}P,
  $$
  with $\tau(P)=\frac{n-1}{n}$, see
  \cite[Corollary~14.13]{NicaSpeicher:2006}.   
\end{proof}

\begin{Rem}
  In the paper \cite{NicaSpeicher:1996} of Nica and Speicher cited above,
  it was shown that for every probability measure $\mu$ there
  is a convolution semigroup $\{\mu^{\boxplus t}\mid t\geq 1\}$.
  Denote $\psi(\mu) = \inf\{t \mid \mu^{\boxplus t}\text{ exists}\}$.
  This can be seen as some kind of ``measure of free non-infinite divisibility''
  in the sense that $\mu$ is freely infinitely divisible if and only if
  $\psi(\mu)=0$. It is related to the
  free divisibility indicator $\phi(\mu)$ of \cite{BelinschiNica:2008} 
  by the inequality $\psi(\mu)\leq 1- \phi(\mu)$.
  If $\tilde{X}$ exists,
  the preceding proof shows that $\psi(Z)\leq \frac{n-1}{n}\psi(\tilde{X})$
  and in particular, if $\tilde{X}$ is $\boxplus$-infinitely divisible,
  then so is $Z$. 
  It then follows from 
  \cite[Theorem 6.1]{ArizmendiHasebeSakuma:2013} 
  that $Z^2$ is freely infinitely divisible as well and consequently also $Q_n$.

  However if $X$ is not freely infinitely divisible,
  the symmetrization $\tilde{X}$ constructed in Corollary~\ref{cor:KrQn}
  in general cannot be realized as an operator,
  see \cite[Remark~12~(2)]{NicaSpeicher:1998}.
\end{Rem}

  We show in the final section that any quadratic form in free even 
  random variables preserves free infinite divisibility.

\subsection{Preservation of free infinite divisibility}
It is shown in \cite{ArizmendiHasebeSakuma:2013}  that the free commutator of
freely infinitely
divisible random variables is also freely infinitely divisible
and the authors ask whether there
are other noncommutative polynomials which preserve free infinite divisibility. 
We show here that for self-adjoint operators this is the case
for any quadratic form in free random variables
whose distribution does not depend on the odd cumulants
of the original distribution.
This includes the free commutator and free sample variance.
In the proof below we will use compound free Poisson distributions $\mu$ 
with rate $\lambda$ and jump distribution $\nu$ which is
the unique probability distribution with free cumulants
$K_n(\mu)=\lambda m_n(\nu)$.
Compound free Poisson distributions are freely infinitely divisible, 
and moreover, any freely infinitely
divisible probability measure is the weak limit distribution of a sequence of
compound free Poisson random variables, see
\cite[Proposition A.2]{ArizmendiHasebeSakuma:2013}.

\begin{prop} \label{prop:NieskonczonaPodzielnosc}
  Let $\X_1, \X_2,\dots, \X_n\in \A_{sa}$ be a free family  
  of even freely infinitely divisible random variables.
  Let  $A=[a_{i,j}]_{i,j=1}^n\in M_n(\IC)$ be a selfadjoint matrix, 
  then the distribution of the quadratic form
  $T_n=\sum_{i,j}^na_{i,j}\X_i\X_j$ is also freely infinitely divisible.
\end{prop}
\begin{proof}
Suppose first that each $X_i$ is a symmetric compound free Poisson variable
with rate $\lambda_i$ and jump distribution $\nu_i$. Let $\Y_i$ be free random
variables with compound free Poisson distribution of rate $\lambda_i$ and jump
distribution $\nu_i^2$, respectively, i.e., with cumulants given by
$K_{r}(\Y_i)=K_{2r}(\X_i)$. Since the support of the jump distribution
$\nu_i^2$ is contained in the positive real axis, it is the spectral
measure of some positive random variable $Z_i$ and follows from
\cite[Proposition~12.18]{NicaSpeicher:2006} that we can represent $Y_i$ as 
a free compression  $Y_i=SZ_iS$ by a free semicircular element $S$ and in
particular $Y_i$ is positive as well.

Using the equation  \eqref{eq:kumulantsamplevariancenotiid} we have
 \begin{align*}  
 K_r(T_n)&=
\sum_{i_1,\dots,i_r\in[n]} 
            \Tr(AE_{i_1}AE_{i_2}\dots AE_{i_r})\,
            \sum_{\substack{ \pi\in \NCeven(2r)\\ 
                \pi \vee \onetwo{r}=\hat{1}_{2r}}}
            K_\pi(X_{i_r},X_{i_1},X_{i_1},X_{i_2},\dots,X_{i_{r-1}},X_{i_r});
& \intertext{Now the bijection introduced in Lemma~\ref{lemm:lematoparzystych}
  implies that $K_{r}(\Y_i)=K_{2r}(\X_i)$ and thus the above is equal to}
 &= \sum_{ i_1,\dots,i_r\in[n]}\Tr(AE_{i_1}AE_{i_2}\dots AE_{i_r})\sum_{\substack{ \pi\in \NC(r) \\ \pi\leq \ker\underline{i} 
 } } K_\pi(Y_{i_1},Y_{i_2},\dots,Y_{i_r})
\\
 &= \sum_{ i_1,\dots,i_r\in[n]}\Tr(AE_{i_1}AE_{i_2}\dots AE_{i_r})\sum_{\pi\in \NC(r)  } K_\pi(Y_{i_1},Y_{i_2},\dots,Y_{i_r})
\\
 &= \sum_{ i_1,\dots,i_r\in[n]}
    \Tr(AE_{i_1}AE_{i_2}\dots AE_{i_r})
    \,
    \tau \Big(\prod_{j=1}^r\Y_{i_j}\Big)=n\times \Tr_n\otimes \tau\left[ \Big(\sum_{i=1}^nAE_i\otimes \Y_{i}\Big)^r\right].
\end{align*}
Hence the cumulant sequence of $T_n$ is the moment sequence of $(A\otimes I)Y$
in the noncommutative probability space $M_n(\C)\otimes \A$, with state
$\Tr_n\otimes \tau$ where ${\mathbf{Y}}=\sum_{i=1}^nE_i\otimes \Y_{i}$.
This operator is not self-adjoint, yet the following
arguments show that it is indeed a positive definite moment sequence.
We have seen above that all  $Y_i$ are  positive and it follows that ${\mathbf{Y}}$ is
positive as well, thus the sequence
$$
\Tr\otimes\tau( ((A\otimes I){\mathbf{Y}} )^r)
=\Tr\otimes\tau( ({\mathbf{Y}}^{1/2}(A\otimes I){\mathbf{Y}}^{1/2} )^r)
$$
is indeed the moment sequence of a self-adjoint random variable.

Suppose now that $X_i$ has a more general symmetric distribution $\mu_i$.
Then the argument of the proof of Proposition A.2. in \cite{ArizmendiHasebeSakuma:2013}
shows that  $\mu_i$ can be approximated by symmetric compound free Poisson variables, 
say $\mu_i=\lim_{k\to \infty}\mu_{i,k}$. 
It follows from the above argument that the distribution $T_n$ can be approximated by 
freely infinitely divisible distributions and since $ID(\boxplus)$ is closed under
convergence in distribution, $T_n$ is freely infinitely divisible as well.
\end{proof}
Putting together Lemma~\ref{lem:oddcancellation} and Proposition~\ref{prop:NieskonczonaPodzielnosc} we obtain the following corollary.
\begin{cor} 
Let $\X_1, \X_2,\dots, \X_n\in \A_{sa}$ be a free family  
  of freely infinitely divisible random variables. 
  Let $P$ be a selfadjoint symmectric polynomial of degree $2$
  in noncommuting variables such that the distribution of
  the random variable $Y=P(X_1,X_2,\dots,X_n)$ does not depend on the odd cumulants.
  Then the distribution of $Y$ is freely infinitely divisible as well.
  In particular,
  the commutator $i(X_1X_2-X_2X_1)$ of two freely infinitely divisible 
  random variables is freely infinitely divisible and
  the same is true of the sample variance of a free identically distributed 
  family of  freely infinitely divisible random variables. 
\end{cor}

\section{Concluding Remarks}
In the present paper we have  shown that the sample variance
shares the following properties with the free commutator:
\begin{enumerate}
 \item Odd cumulants do not contribute to the distribution.
 \item Free infinite divisibility is preserved.
\end{enumerate}

This phenomenon raises the following problems and conjectures,
some of which will be investigated in forthcoming papers.
\begin{problem}
  \label{prob:oddpoly}
  Characterize the class of selfadjoint polynomials 
  $P\in\IC\langle X_1,X_2,\dots,X_n\rangle$ 
  in noncommuting variables $X_1,X_2,\dots,X_n$ with the property
  that the distribution of $P(\X_1,\dots, \X_n)$ does not depend on the
  odd cumulants of $X$ whenever $\X_1, \X_2,\dots, \X_n$ 
  are free copies of a fixed random variable $X$.
\end{problem}

\begin{Con} 
  Whenever a homogeneous polynomial $P$ has the properties described 
  in Problem~\ref{prob:oddpoly}
  and $X_1,X_2,\dots,X_n$ are free copies of a freely infinitely divisible
  random variable $X$, then $P(\X_1,\dots, \X_n)$ is
  freely infinitely divisible as well.
\end{Con}

\emph{Acknowledgments}.
The authors would like to thank Marek Bo\.zejko and
Roland Speicher for several discussions and helpful comments.
We are very grateful to Takahiro Hasebe for many comments and in particular
for pointing out a gap in the proof of Proposition~\ref{prop:NieskonczonaPodzielnosc}.
The first author also thanks Abram Kagan for a very interesting discussion about Ruben's
problem during AMISTAT 2015 in Prague.
Finally we thank the referee for a careful reading of the manuscript and
numerous minor corrections.

The work was partially supported by grant number 2014/15/B/ST1/00064 from the \textit{Narodowe Centrum Nauki}, 
 Project No P 25510-N26 of the Austrian Science Fund (FWF) and 
travel grant PL 08/2016 of the oead.
\bibliography{ChiSquare}
\bibliographystyle{amsplain}

\end{document}